\documentclass{article}

\usepackage{amsmath,amsthm,amsopn,amstext,amscd,amsfonts,amssymb,verbatim}
\usepackage[numbers]{natbib}

\usepackage{graphicx}
\usepackage{verbatim,color}
\usepackage{caption}
\usepackage{subcaption}

\def\tr{\mathop{\rm tr}\nolimits}
\def \Vol {\mathop{\rm Vol}\nolimits}
\def\re{\mathop{\rm Re}\nolimits}
\def\etr{\mathop{\rm etr}\nolimits}
\def\diag{\mathop{\rm diag}\nolimits}

\def\vec{\mathop{\rm vec}\nolimits}

\renewenvironment{abstract}
                 {\vspace{6pt}
                  \begin{center}
                  \begin{minipage}{5in}
                  \centerline{\textbf{Abstract}}
                  \noindent\ignorespaces
                 }
                 {\end{minipage}\end{center}}
                 
\newtheorem{theorem}{\textbf{Theorem}}[section]
\newtheorem{corollary}{\textbf{Corollary}}[section]

\newtheorem{proposition}{\textbf{Proposition}}[section]
\theoremstyle{definition}
\newtheorem{definition}{\textbf{Definition}}[section]

\title{\Large \textbf{Multimatricvariate and multimatrix variate distributions based on elliptically contoured laws under real normed division algebras}}
\author{
  \textbf{Jos\'e A. D\'{\i}az-Garc\'{\i}a} \thanks{Corresponding author\newline
    }\\
  {\normalsize Universidad Aut\'onoma de Chihuahua} \\
  {\normalsize Facultad de Zootecnia y Ecolog\'{\i}a} \\
  {\normalsize Perif\'erico Francisco R. Almada Km 1, Zootecnia} \\
  {\normalsize 33820 Chihuahua, Chihuahua, M\'exico}\\
  {\normalsize E-mail: jadiaz@uach.mx}\\
  \textbf{Francisco J. Caro-Lopera}\\
  {\normalsize University of Medellin} \\
  {\normalsize Faculty of Basic Sciences} \\
  {\normalsize Carrera 87 No.30-65} \\
  {\normalsize Medell\'{\i}n, Colombia} \\
  {\normalsize E-mail: fjcaro@udemedellin.edu.co} \\[2ex]
}
\date{}
\begin{document}
\maketitle

\begin{abstract}
This paper proposes famillies of multimatricvariate and multimatrix variate distributions based on elliptically contoured laws in the context of real normed division algebras. The work allows to answer the following inference problems about random matrix variate distributions: 1) Modeling of two or more probabilistically dependent random variables in all possible combinations whether univariate, vector and matrix simultaneously. 2) Expected marginal distributions under independence and joint estimation of models under likelihood functions of dependent samples. 3) Definition of a likelihood function for dependent samples in the mentioned random dimensions and under real normed division algebras. The corresponding real distributions are alternative approaches to the existing univariate and vector variate copulas, with the additional advantages previously listed. An application for quaternionic algebra is illustrated by a computable dependent sample joint distribution for landmark data emerged from shape theory.
\end{abstract}                 

\section{Introduction}\label{sec:1}

A postmodern world of increasingly heuristic and interdisciplinary problems requires integrative solutions that do not always have the answer in the current frameworks of mathematics and statistics. Chaos theory as a paradigmatic element of phenomenal interrelation links more and more random variables. They are not only univariate and vector variables, but matrix variates, due to their great versatility in the description of variability and intrinsic multiple correlation. Likewise, the domain of the variables began to move from the real field to the remaining real normed division algebras, such as complex, quaternionic and octonionic. There the emerging models of multiple natural, exact and engineering sciences have their own scope and they are not restricted by the real field. Given multiple statistical models for marginal distributions in univariate and vector cases, the additional problem of parametric estimation arises from a set of samples provided by the  scientific expert user of statistics. Then comes another challenge for statistics, which has historically been resolved by optimising likelihood functions as a joint law for sample distribution. The theory of copulas emerged as a possible solution to the problem of explaining a joint phenomenon of known interrelation. Countless link functions then appear for two variables, however, the intention of a dependent relationship is diluted when using the typical likelihood functions for estimating copula parameters from independent samples. The role of likelihood in the history of statistics is of such transcendence that it is a popular tool for estimating parameters; however, its profound simplicity based on sample independence is never debated. While copulas and similar theories represent the solution to the problem of dependence, the concept of likelihood function, defined as the product of the marginal densities, seems immutable and universal. Recently, the authors have begun the discussion of the role of classical likelihood functions in statistics, showing the differences in application in databases declared dependently probabilistic. Time series in particular are the source of the greatest discrepancy, given the historical effort in proposing models for variance and volatility. But maintaining the postulate of an estimate via likelihood functions on independent data is an obvious contradiction in the face of a temporal process that intrinsically is founded on dependency. A work that launched this new way of reframing likelihood functions appeared recently in \citet{dgclpr:22}. In a real data series, the discrepancy of likelihood function over independent time samples versus realistic likelihood function over dependent samples was observed. The study indicated that decisions diffused about the mean estimated from the database using independent likelihood were in the tails of the dependent expected likelihood distribution. The finding was possible thanks to the definition of the so-termed multivector variate distribution, a natural way of defining a likelihood function of dependent vector samples, further parameterised by a large class of elliptically contoured distributions  that are isolated from the choices of popular models. Until then, the mutivector could be an approach to vector copulas, but the implemented theory is designed to address the matrix case, which is still an open problem for copula theory. Then, the matrix version appeared in the form of the so-termed classes of multimatricvariate distributions, which provided the likelihood function or joint matrix distributions with a family of elliptically contoured distributions  with a variety of kurtosis and symmetries, see \citet{dgcl:22}. The choice of distributions that are invariant under the class of elliptically contoured laws within the family of multimatricvariate and multimatrix variate distributions gives robustness to the analysis since it avoids the difficult paradigm of adjusting the link function, if we speak in parallel to the procedure. in vector copulas. 

Multimatricvariate distributions specialise in distributions based on the determinant. Still to be defined a class of multiple distributions of random matrices that depended on the trace, the other function that has historically governed the theory of matrix distributions. Recently, the so-termeded multimatrix distributions appear in \citet{dgcl:24}. Like the multimatricvariate distributions, they maintain the philosophy of probabilistic dependence, computability and a wide class of underlying distributions. 

When we surpass the level of random vectors, where the most popular statistical techniques remain, such as copulas, we find immense difficulty for the calculation.  We must observe the domain of the matrices, their transformations and Jacobians, and therefore their integration. It is generally difficult because it involves averaging over orthogonal groups, cones and hypercubes. The central, isotropic, non-central and non-isotropic cases also emerge. These integrations comprises the geometric filtering that comes for their definitions and factorisations. 

At this point we find that the new multimatricvariate and multimatrix variate distributions on the real field satisfy the first 2 of the 3 conditions that we establish as a paradigm for a robust point of view of joint distributions, namely: 1) modeling of several dependent matrices capable of being combined with all univariate and/or vector variables governed by elliptically contoured models; 2) with strictly verifiable marginals under the trivial case of independence and testable from the  dependent sample likelihood; and 3) founded on the same algebraic principles that allow them to be applied indistinctly from the real normed division algebras.

We now address the third desirable characteristic: its versatility in the four unique real normed division algebras. 

Until a couple of decades ago it would be unthinkable to unify the theory of real random matrices with complex ones. History shows us that the theory of distributions of the central cases overflowed into an immense effort to characterize the indispensable elements of random matrices, namely: After its establishment, works on the complex case began to appear very slowly and completely conceptually distant from their real analogues.  The so-termed statistical theory under real normed division algebras is to statistics what the expected theory of unification of forces is to physics. Its simplification is surprising in that all its results are reduced to modifying a beta parameter that goes from 1 (real) to 2 (complex) to 4 (quaternionic) to 8 (octonionic), see \citet{{dg:14}} were the transitions to complex, quaternionic and octonion are reached by changing the support group from orthogonal to unitary,  compact symplectic or exceptional type. 

If the real multimatricvariate distributions \citep{dgcl:22} and  multimatrix variate distributions \citep{dgcl:24} articles are parallel compared with the results of the present work, an  apparent repetition shall appear by its simplicity; Its notation makes the beta parameter open one of the four universes by simply changing the value, although understanding the profound difference that underpins them requires an extensive and difficult literature that began in other areas far from statistics. The non commutativity for quaternions and non associative for octonions promotes a deep research for some applications in those algebras.

We place the above discussion into the setting of two main problems that we shall address in this article.

The interest in multimatricvariate  and multimatrix variate distributions has been motivated by the following two situations:
\begin{enumerate}
  \item In different areas of knowledge (such as Finance and Hydrology, among others), people are interested in simultaneously modeling two random variables, say $X$ and $Y$, which are suspected of not being probabilistically independent. On the one hand, the marginal distributions of each variable are known, whether they are $f_{X}(x)$ and $g_{Y}(y)$. Typically this problem has been approached assuming that the random variables $X$ and $Y$ are independent and, as a function of joint density of the two-dimensional vector $(X,Y)'$, the product of the marginals, $r_{X,Y}(x,y) = f_{X}(x)g_{Y}(y)$, has been considered. Thus, for example, in this case the likelihood function, given the two-dimensional sample $(x_{1},y_{1}), \cdots,(x_{k},y_{k})$, denoted as $L(\boldmath{\theta};(x_{1},y_{1}), \cdots,(x_{k},y_{k}))$ is defined as
      \begin{eqnarray*}
        L(\mathbf{\theta};(x_{1},y_{1}), \cdots,(x_{n},y_{n})) &=& \prod_{j=1}^{k} r_{X_{j},Y_{j}}(x_{j},y_{j})\\
         &=& \prod_{j=1}^{k} f_{X_{j}}(x_{j})g_{Y_{j}}(y_{j})
      \end{eqnarray*}
For some parameter vector $\boldmath{\theta} \in \Re^{p}$ which is part of the density $r_{X,Y}(x,y)$.
      
       Alternatively, based on variable changes on a set of independent random variables, the random vector $(X,Y)'$ was generated, where now the variables $X$ and $Y$ are not independent and their density function is known joint $t_{X,Y}(x,y) \neq f_{X}(x)g_{Y}(y)$, such that
      $$
        f_{X}(x)= \int_{\mathfrak{\Re}}t_{X,Y}(x,y)(dy) {\mbox{\hspace{0.5cm} and \hspace{0.5cm}}}  g_{Y}(y)= \int_{\Re}t_{X,Y}(x,y)(dx).
      $$
Under this approach we have that
       $$
         L(\boldmath{\theta};(x_{1},y_{1}),\cdots,(x_{k},y_{k})) = \prod_{j=1}^{k}t_{ X_{j},Y_{j}}(x_{j},y_{j})
       $$
       See \citet{ln:82}, \citet{cn:84}, \citet{ol:03}, \citet{n:07,n:13} and \citet{spj:14} among many others.
          
       This situation also occurs in other multivariate problems. Then, parallel solutions were proposed in the vector and matrix cases, giving foothold to the study of bi-matrix variate distributions in the real and complex cases. In the last case joint distributions of random matrices, say $\mathbf{X}$ and $\mathbf{Y}$ dependent with joint density function, termed bimatrix variate distribution, are proposed such that the marginal densities of $\mathbf{X}$ and $\mathbf{Y}$ are the usual assumptions, see \citet{or:64}, \citet{dggj:10a, dggj:10b,dggj:11b}, \citet{brea:11}, and \cite{e:11} and references therein.
  \item In another case, we are interested in defining the likelihood function by a joint function of the sample, but which is not defined as the product of the marginals, that is, the sample is not independent. In the univariate problem, one answer considers the elliptically contoured distribution of the vector $(X_{1}, \dots,X_{k})'$ as a likelihood function, noting that in reality the elliptically contoured distribution actually defines a distribution family, see \citet{fzn:90}, \citet{fz:90}, \citet{gv:93} and the reference therein. 
      
 Based on the family of matrix variate elliptically contoured distributions, the multimatrix variate and multimatricvariate distributions were proposed as a generalisation of the bi-matrix variate distributions, which are defined as the joint distribution of the dependent random matrices $\mathbf{X}_{1 }, \dots,\mathbf{X}_{k}$, see \citet{dgclpr:22}, and \citet{dgcl:22,dgcl:24}. Thus, the multimatrix variate and multimatricvariate distributions can be used as likelihood functions for a sample of dependent random matrices with certain (usual) marginal distributions. Thus, the likelihood function of the sample $\mathbf{X}_{1}, \dots,\mathbf{X}_{k}$ is defined as
       $$
         L(\mathbf{\Theta};\mathbf{X}_{1}, \dots,\mathbf{X}_{k}) = f_{\mathbf{X}_{1}, \dots,\mathbf {X}_{k}}(\mathbf{X}_{1}, \dots,\mathbf{X}_{k}).
       $$
\end{enumerate}
Under the theory of multimatrix matrix or multimatricvariate distributions, each matrix $\mathbf{X}_{j}$, $j= 1,\dots,k$ into the density $f_{\mathbf{X}_{1}, \dots,\mathbf{X}_{k}}(\mathbf{X}_{1}, \dots,\mathbf{X} _{k})$ can follow a different marginal distribution. This answers the following problem under an independent or dependent samples: Suppose that we have a matrix random sample as follows:

\bigskip
\begin{center}
  \begin{tabular}{c|cccc}
  \hline
   \hline 
     $\mathbf{X}_{1}$ & $\mathbf{X}_{11}$ & $\mathbf{X}_{12}$ & $\cdots$ & $\mathbf{X}_{1r}$ \\
     $\mathbf{X}_{2}$ & $\mathbf{X}_{21}$ & $\mathbf{X}_{22}$ & $\cdots$ & $\mathbf{X}_{2r}$ \\
     $\vdots$ & $\vdots$ & $\vdots$ & $\ddots$ & $\vdots$ \\
     $\mathbf{X}_{k}$ & $\mathbf{X}_{k1}$ & $\mathbf{X}_{k2}$ & $\cdots$ & $\mathbf{X}_{kr}$ \\
  \hline
\end{tabular}
\end{center}
\medskip
And assume that the matrix $\mathbf{\Theta}$ contains the parameters of interest. Then the likelihood function $L(\mathbf{\Theta}; \cdot)$ can be defined as:
$$
\left\{
  \begin{array}{ll}
    \displaystyle\prod_{j=1}^{r}f_{\mathbf{X}_{1j}, \dots,\mathbf{X}_{kj}}(\mathbf{X}_{1j}, \dots,\mathbf{X}_{kj}), & \hbox{independence,} \\
    f_{\mathbf{X}_{11}, \dots,\mathbf{X}_{1r},\dots,\mathbf{X}_{k1}, \dots,\mathbf{X}_{kr}}(\mathbf{X}_{11}, \dots,\mathbf{X}_{1r},\dots,\mathbf{X}_{k1}, \dots,\mathbf{X}_{kr}), & \hbox{dependence.}
  \end{array}
\right.
$$

In the bi-matrix variate case, these problems have been studied in the real and complex cases, each giving rise to a series of non correlated publications. Fortunately, in terms of the theory of the real normed division algebras, a unification of the real and complex cases is possible. And an extension to the quaternionic and octonionic algebras is also feasible. It is worth mentioning that the octornionic case is still under research. At this time, they are valid for $2 \times 2$ octornionic matrices  and in general it can only be conjectured that they may be valid.

In the present work, the multimatrix variate and multimatricvariate  distributions are studied for matrix arguments which elements belong to the real normed division algebras. A brief description of the notation and some Jacobians for real normed division algebras is presented in Section \ref{sec:2}. In addition, two more Jacobians are obtained and the definition of the matrix variate elliptically contoured distribution for real normed division algebras is presented. The main results on multimatrix variate and multimatricvariate distributions for real normed division algebras are obtained in Section \ref{sec:3}. Some properties and extensions of multimatrix variate and multimatricvariate distribution with more than two different types of distributions in their arguments are studied in Section \ref{sec:4}. An example in the quaternionic case is full derived in Section \ref{sec:5}.

\section{Notation and preliminary results}\label{sec:2}

A detailed discussion of real normed division algebras may be found in \citet{b:02}. For
convenience, we shall introduce some notations, although in general we adhere to standard
notations.

A \textbf{vector space} is always a finite-dimensional module over the field
of real numbers. An \textbf{algebra} $\mathfrak{F}$ is a vector space that is equipped with a
bilinear map $m: \mathfrak{F} \times \mathfrak{F} \rightarrow \mathfrak{F}$ termed
\emph{multiplication} and a nonzero element $1 \in \mathfrak{F}$ termed the \emph{unit} such
that $m(1,a) = m(a,1) = 1$. As usual, we abbreviate $m(a,b) = ab$ as $ab$. We do not assume
$\mathfrak{F}$ associative. Given an algebra, we freely think of real numbers as elements of
this algebra via the map $\omega \mapsto \omega 1$.

An algebra $\mathfrak{F}$ is a \textbf{division algebra} if given $a, b \in \mathfrak{F}$ with
$ab=0$, then either $a=0$ or $b=0$. Equivalently, $\mathfrak{F}$ is a division algebra if the
operation of left and right multiplications by any nonzero element is invertible. A
\textbf{normed division algebra} is an algebra $\mathfrak{F}$ that is also a normed vector
space with $||ab|| = ||a||||b||$. This implies that $\mathfrak{F}$ is a division algebra and
that $||1|| = 1$.

There are exactly four normed division algebras: real numbers ($\Re$), complex numbers
($\mathfrak{C}$), quaternions ($\mathfrak{H}$) and octonions ($\mathfrak{O}$), see
\citet{b:02}. Taking into account that $\Re$, $\mathfrak{C}$, $\mathfrak{H}$ and
$\mathfrak{O}$ are the only normed division algebras; moreover, they are the only alternative
division algebras, and all division algebras have a real dimension of $1, 2, 4$ or $8$, which
is denoted by $\beta$, see \citet[Theorems 1, 2 and 3]{b:02}. In other branches of mathematics,
the parameter $\alpha = 2/\beta$ is used, see \citet{er:05}.

Let ${\mathcal L}^{\beta}_{m,n}$ be the linear space of all $n \times m$ matrices of rank $m
\leq n$ over $\mathfrak{F}$ with $m$ distinct positive singular values, where $\mathfrak{F}$
denotes a \emph{real finite-dimensional normed division algebra}. Let $\mathfrak{F}^{n \times
m}$ be the set of all $n \times m$ matrices over $\mathfrak{F}$. The dimension of
$\mathfrak{F}^{n \times m}$ over $\Re$ is $\beta mn$. Let $\mathbf{A} \in \mathfrak{F}^{n
\times m}$, then $\mathbf{A}^{H} = \overline{\mathbf{A}}^{T}$ denotes the usual conjugate
transpose.

The set of matrices $\mathbf{H}_{1} \in \mathfrak{F}^{n \times m}$ such that
$\mathbf{H}_{1}^{H}\mathbf{H}_{1} = \mathbf{I}_{m}$ is a manifold denoted ${\mathcal
V}_{m,n}^{\beta}$, is termed the \emph{Stiefel manifold} ($\mathbf{H}_{1}$ is also known as
\emph{semi-orthogonal} ($\beta = 1$), \emph{semi-unitary} ($\beta = 2$), \emph{semi-symplectic}
($\beta = 4$) and \emph{semi-exceptional type} ($\beta = 8$) matrices, see \citet{dm:99}). The
dimension of $\mathcal{V}_{m,n}^{\beta}$ over $\Re$ is $[\beta mn - m(m-1)\beta/2 -m]$. In
particular, ${\mathcal V}_{m,m}^{\beta}$ with dimension over $\Re$, $[m(m+1)\beta/2 - m]$, is
the maximal compact subgroup $\mathfrak{U}^{\beta}(m)$ of ${\mathcal L}^{\beta}_{m,m}$ and
consists of all matrices $\mathbf{H} \in \mathfrak{F}^{m \times m}$ such that
$\mathbf{H}^{H}\mathbf{H} = \mathbf{I}_{m}$. Therefore, $\mathfrak{U}^{\beta}(m)$ is the
\emph{real orthogonal group} $\mathcal{O}(m)$ ($\beta = 1$), the \emph{unitary group}
$\mathcal{U}(m)$ ($\beta = 2$), \emph{compact symplectic group} $\mathcal{S}p(m)$ ($\beta = 4$)
or \emph{exceptional type matrices} $\mathcal{O}o(m)$ ($\beta = 8$), for $\mathfrak{F} = \Re$,
$\mathfrak{C}$, $\mathfrak{H}$ or $\mathfrak{O}$, respectively.

We denote by ${\mathfrak S}_{m}^{\beta}$ the real vector space of all $\mathbf{S} \in
\mathfrak{F}^{m \times m}$ such that $\mathbf{S} = \mathbf{S}^{H}$. Let
$\mathfrak{P}_{m}^{\beta}$ be the \emph{cone of positive definite matrices} $\mathbf{S} \in
\mathfrak{F}^{m \times m}$; then $\mathfrak{P}_{m}^{\beta}$ is an open subset of ${\mathfrak
S}_{m}^{\beta}$. Over $\Re$, ${\mathfrak S}_{m}^{\beta}$ consist of \emph{symmetric} matrices;
over $\mathfrak{C}$, \emph{Hermitian} matrices; over $\mathfrak{H}$, \emph{quaternionic
Hermitian} matrices (also termed \emph{self-dual matrices}) and over $\mathfrak{O}$,
\emph{octonionic Hermitian} matrices. Generically, the elements of $\mathfrak{S}_{m}^{\beta}$
are termed
 \textbf{Hermitian matrices}, irrespective of the nature of $\mathfrak{F}$. The
dimension of $\mathfrak{S}_{m}^{\beta}$ over $\Re$ is $[m(m-1)\beta+2m]/2$.

Let $\mathfrak{D}_{m}^{\beta}$ be the \emph{diagonal subgroup} of $\mathcal{L}_{m,m}^{\beta}$
consisting of all $\mathbf{D} \in \mathfrak{F}^{m \times m}$, $\mathbf{D} = \diag(d_{1},
\dots,d_{m})$.

For any matrix $\mathbf{X} \in \mathfrak{F}^{n \times m}$, $d\mathbf{X}$ denotes the\emph{
matrix of differentials} $(dx_{ij})$. Finally, we define the \emph{measure} or volume element
$(d\mathbf{X})$ when $\mathbf{X} \in \mathfrak{F}^{m \times n}, \mathfrak{S}_{m}^{\beta}$,
$\mathfrak{D}_{m}^{\beta}$ or $\mathcal{V}_{m,n}^{\beta}$, see \citet{d:02}.

If $\mathbf{X} \in \mathfrak{F}^{n \times m}$ then $(d\mathbf{X})$ (the Lebesgue measure in
$\mathfrak{F}^{n \times m}$) denotes the exterior product of the $\beta mn$ functionally
independent variables
$$
  (d\mathbf{X}) = \bigwedge_{i = 1}^{n}\bigwedge_{j = 1}^{m}dx_{ij} \quad \mbox{ where }
    \quad dx_{ij} = \bigwedge_{r = 1}^{\beta}dx_{ij}^{(r)}.
$$

If $\mathbf{S} \in \mathfrak{S}_{m}^{\beta}$ (or $\mathbf{S} \in \mathfrak{T}_{L}^{\beta}(m)$)
then $(d\mathbf{S})$ (the Lebesgue measure in $\mathfrak{S}_{m}^{\beta}$ or in
$\mathfrak{T}_{L}^{\beta}(m)$) denotes the exterior product of the $m(m+1)\beta/2$ functionally
independent variables (or denotes the exterior product of the $m(m-1)\beta/2 + m$ functionally
independent variables, if $s_{ii} \in \Re$ for all $i = 1, \dots, m$)
$$
  (d\mathbf{S}) = \left\{
                    \begin{array}{ll}
                      \displaystyle\bigwedge_{i \leq j}^{m}\bigwedge_{r = 1}^{\beta}ds_{ij}^{(r)}, &  \\
                      \displaystyle\bigwedge_{i=1}^{m} ds_{ii}\bigwedge_{i < j}^{m}\bigwedge_{r = 1}^{\beta}ds_{ij}^{(r)}, &
                       \hbox{if } s_{ii} \in \Re.
                    \end{array}
                  \right.
$$
The context generally establishes the conditions on the elements of $\mathbf{S}$, that is, if
$s_{ij} \in \Re$, $\in \mathfrak{C}$, $\in \mathfrak{H}$ or $ \in \mathfrak{O}$. It is
considered that
$$
  (d\mathbf{S}) = \bigwedge_{i \leq j}^{m}\bigwedge_{r = 1}^{\beta}ds_{ij}^{(r)}
   \equiv \bigwedge_{i=1}^{m} ds_{ii}\bigwedge_{i < j}^{m}\bigwedge_{r =
1}^{\beta}ds_{ij}^{(r)}.
$$
Observe, too, that for the Lebesgue measure $(d\mathbf{S})$ defined thus, it is required that
$\mathbf{S} \in \mathfrak{P}_{m}^{\beta}$, that is, $\mathbf{S}$ must be a non singular
Hermitian matrix (Hermitian positive definite matrix).

If $\mathbf{\Lambda} \in \mathfrak{D}_{m}^{\beta}$ then $(d\mathbf{\Lambda})$ (the Legesgue
measure in $\mathfrak{D}_{m}^{\beta}$) denotes the exterior product of the $\beta m$
functionally independent variables
$$
  (d\mathbf{\Lambda}) = \bigwedge_{i = 1}^{n}\bigwedge_{r = 1}^{\beta}d\lambda_{i}^{(r)}.
$$
If $\mathbf{H}_{1} \in \mathcal{V}_{m,n}^{\beta}$ then
$$
  (\mathbf{H}^{H}_{1}d\mathbf{H}_{1}) = \bigwedge_{i=1}^{n} \bigwedge_{j =i+1}^{m}
  \mathbf{h}_{j}^{H}d\mathbf{h}_{i}.
$$
where $\mathbf{H} = (\mathbf{H}_{1}|\mathbf{H}_{2}) = (\mathbf{h}_{1}, \dots,
\mathbf{h}_{m}|\mathbf{h}_{m+1}, \dots, \mathbf{h}_{n}) \in \mathfrak{U}^{\beta}(m)$. It can be
proved that this differential form does not depend on the choice of the $\mathbf{H}_{2}$
matrix. When $m = 1$; $\mathcal{V}^{\beta}_{1,n}$ defines the unit sphere in
$\mathfrak{F}^{n}$. This is, of course, an $(n-1)\beta$- dimensional surface in
$\mathfrak{F}^{n}$. When $m = n$ and denoting $\mathbf{H}_{1}$ by $\mathbf{H}$,
$(\mathbf{H}^{H}d\mathbf{H})$ is termed the \emph{Haar measure} on $\mathfrak{U}^{\beta}(m)$.

The surface area or volume of the Stiefel manifold $\mathcal{V}^{\beta}_{m,n}$ is
\begin{equation}\label{vol}
    \Vol(\mathcal{V}^{\beta}_{m,n}) = \int_{\mathbf{H}_{1} \in
  \mathcal{V}^{\beta}_{m,n}} (\mathbf{H}^{H}_{1}d\mathbf{H}_{1}) =
  \frac{2^{m}\pi^{mn\beta/2}}{\Gamma^{\beta}_{m}[n\beta/2]},
\end{equation}
where $\Gamma^{\beta}_{m}[a]$ denotes the multivariate Gamma function for the space
$\mathfrak{S}_{m}^{\beta}$, and is defined by
\begin{eqnarray*}
  \Gamma_{m}^{\beta}[a] &=& \displaystyle\int_{\mathbf{A} \in \mathfrak{P}_{m}^{\beta}}
  \etr\{-\mathbf{A}\} |\mathbf{A}|^{a-(m-1)\beta/2 - 1}(d\mathbf{A}) \\
&=& \pi^{m(m-1)\beta/4}\displaystyle\prod_{i=1}^{m} \Gamma[a-(i-1)\beta/2],
\end{eqnarray*}
where $\etr(\cdot) = \exp(\tr(\cdot))$, $|\cdot|$ denotes the determinant and $\re(a)
> (m-1)\beta/2$, see \citet{gr:87}. If $\mathbf{A}\in \mathcal{L}_{m,n}^{\beta}$ then by $\vec (\mathbf{A})$ we mean the $mn \times 1$ vector formed by stacking the columns of
$\mathbf{A}$ under each other; that is, if $\mathbf{A} = [\mathbf{a}_{1}\mathbf{a}_{2}\dots
\mathbf{a}_{m}]$, where $\mathbf{a}_{j} \in \mathcal{L}_{1,n}^{\beta}$ for $j = 1, 2, \dots,m$
$$
  \vec(\mathbf{A})= \left [
                     \begin{array}{c}
                       \mathbf{a}_{1} \\
                       \mathbf{a}_{2} \\
                       \vdots \\
                       \mathbf{a}_{m}
                     \end{array}
                    \right ].
$$

Below are summarised some Jacobians in terms of the $\beta$ parameter. For a detailed
discussion of this and related issues  see \citet{d:02}, \citet{er:05}, \citet{f:05} and
\citet{k:84}.

\begin{proposition}\label{lem1}
Let $\mathbf{A} \in {\mathcal L}_{n,n}^{\beta}$, $\mathbf{B} \in {\mathcal L}_{m,m}^{\beta}$
and $\mathbf{C} \in {\mathcal L}_{m,n}^{\beta}$ be matrices of constants, $\mathbf{Y}$ and
$\mathbf{X} \in {\mathcal L}_{m,n}^{\beta}$ a matrices of functionally independent variables
such that $\mathbf{Y} = \mathbf{AXB} + \mathbf{C}$. Then
\begin{equation}\label{lt}
    (d\mathbf{Y}) = |\mathbf{A}^{H}\mathbf{A}|^{\beta m/2} |\mathbf{B}^{H}\mathbf{B}|^{\beta
    n/2}(d\mathbf{X}).
\end{equation}
\end{proposition}

\begin{proposition}\label{lems}
Let $\mathbf{S} \in \mathfrak{P}_{m}^{\beta}$. If $\mathbf{Y} =
\mathbf{A}\mathbf{S}\mathbf{A}^{H}$, $\mathbf{A} \in \mathfrak{P}_{m}^{\beta}.$
\begin{equation}\label{inv}
    (d\mathbf{Y}) = |\mathbf{A}^{H}\mathbf{A}|^{\beta(m - 1)/2 + 1}(d\mathbf{X}).
\end{equation}
\end{proposition}

\begin{proposition}[Singular value decomposition, $SVD$]\label{lemsvd}
Assume that $\mathbf{X} \in {\mathcal L}_{m,n}^{\beta}$, such that $\mathbf{X} =
\mathbf{V}_{1}\mathbf{DW}^{H}$ with $\mathbf{V}_{1} \in {\mathcal V}_{m,n}^{\beta}$,
$\mathbf{W} \in \mathfrak{U}^{\beta}(m)$ and $\mathbf{D} = \diag(d_{1}, \cdots,d_{m}) \in
\mathfrak{D}_{m}^{1}$, $d_{1}> \cdots > d_{m} > 0$. Then
\begin{equation}\label{svd}
    (d\mathbf{X}) = 2^{-m}\pi^{\tau} \prod_{i = 1}^{m} d_{i}^{\beta(n - m + 1) -1}
    \prod_{i < j}^{m}(d_{i}^{2} - d_{j}^{2})^{\beta} (d\mathbf{D}) (\mathbf{V}_{1}^{H}d\mathbf{V}_{1})
    (\mathbf{W}^{H}d\mathbf{W}),
\end{equation}
where
$$
  \tau = \left\{
             \begin{array}{rl}
               0, & \beta = 1; \\
               -m, & \beta = 2; \\
               -2m, & \beta = 4; \\
               -4m, & \beta = 8.
             \end{array}
           \right.
$$
\end{proposition}
As a consequence of this result, we have the following statement.
\begin{proposition}\label{lemW}
Let $\mathbf{X} \in {\mathcal L}_{m,n}^{\beta}$, and  $\mathbf{S} = \mathbf{X}^{H}\mathbf{X}
\in \mathfrak{P}_{m}^{\beta}.$ Then
\begin{equation}\label{w}
    (d\mathbf{X}) = 2^{-m} |\mathbf{S}|^{\beta(n - m + 1)/2 - 1}
    (d\mathbf{S})(\mathbf{V}_{1}^{H}d\mathbf{V}_{1}),
\end{equation}
with $\mathbf{V}_{1} \in {\mathcal V}_{m,n}^{\beta}$.
\end{proposition}

\begin{proposition}\label{lemi}
Let $\mathbf{S} \in \mathfrak{P}_{m}^{\beta}.$ Then ignoring the sign, if $\mathbf{Y} =
\mathbf{S}^{-1}$
\begin{equation}\label{inv}
    (d\mathbf{Y}) = |\mathbf{S}|^{-\beta(m - 1) - 2}(d\mathbf{S}).
\end{equation}
\end{proposition}

\begin{theorem}\label{teo1}
Assume that $\mathbf{X}  \in {\mathcal L}_{m,n}^{\beta} $ and $\mathbf{Y}  \in {\mathcal L}_{m,n}^{\beta}$ are matrices of functionally independent variables.
\begin{description}
  \item[i)] Define $\mathbf{Y} = \mathbf{X} (\mathbf{I}_{m}-\mathbf{X}^{H}\mathbf{X})^{-1/2}$. Then
     \begin{equation}\label{eq2}
        (d\mathbf{Y}) = \left|\mathbf{I}_{m} - \mathbf{X}^{H}\mathbf{X}\right|^{-\beta(n+m+1)/2-1} (d\mathbf{X}).
     \end{equation}
  \item[ii)] If $\mathbf{X} = \mathbf{Y}(\mathbf{I}_{m}+\mathbf{Y}^{H}\mathbf{Y})^{-1/2}$, we have
     \begin{equation}\label{eq3}
        (d\mathbf{X}) = \left|\mathbf{I}_{m} + \mathbf{Y}^{H}\mathbf{Y}\right|^{-\beta(n+m+1)/2-1} (d\mathbf{Y}),
     \end{equation}
\end{description}
where $n \geq m$.
\end{theorem}
\begin{proof}
\textbf{i}) Define $\mathbf{A} = \mathbf{Y}^{H}\mathbf{Y} =
(\mathbf{I}_{m}-\mathbf{X}^{H}\mathbf{X})^{-1/2}\mathbf{X}^{H}\mathbf{X}(\mathbf{I}_{m}-\mathbf{X}^{H}\mathbf{X})^{-1/2}$
and $\mathbf{B} = \mathbf{X}^{H}\mathbf{X}$. And observe that
\begin{eqnarray*}
  \mathbf{A} &=& (\mathbf{I}_{m}-\mathbf{X}^{H}\mathbf{X})^{-1/2}\mathbf{X}^{H}\mathbf{X}(\mathbf{I}_{m}-\mathbf{X}^{H}
        \mathbf{X})^{-1/2}\\
    &=& (\mathbf{I}_{m}-\mathbf{B})^{-1/2}\mathbf{B}(\mathbf{I}_{m}-\mathbf{B})^{-1/2} \\
    &=& \left[\mathbf{B}^{-1}(\mathbf{I}_{m} - \mathbf{B})\right]^{-1/2} \left[\mathbf{B}^{-1}(\mathbf{I}_{m}
        - \mathbf{B})\right]^{-1/2}\\
    &=& \left(\mathbf{B}^{-1}- \mathbf{I}_{m}\right)^{-1/2} \left(\mathbf{B}^{-1}- \mathbf{I}_{m} \right)^{-1/2}\\
    &=& \left(\mathbf{B}^{-1}- \mathbf{I}_{m}\right)^{-1} = (\mathbf{I}_{m} - \mathbf{B})^{-1}\mathbf{B} =
    (\mathbf{I}_{m} - \mathbf{B})^{-1}-\mathbf{I}_{m}.
\end{eqnarray*}
Then by (\ref{w}), for $\mathbf{H}_{1},\mathbf{G}_{1} \in \mathcal{V}_{m,n}$, and writing these for $(d\mathbf{A})$ and $(d\mathbf{B})$ we obtain
\begin{eqnarray}
  \label{eq4}(d\mathbf{A}) &=& 2^{m}|\mathbf{A}|^{-\beta(n-m+1)/2+1}(d\mathbf{Y})(\mathbf{H}^{H}_{1}d\mathbf{H}_{1})^{-1} \\
  \label{eq5}(d\mathbf{B}) &=& 2^{m}|\mathbf{B}|^{-\beta(n-m+1)/2+1}(d\mathbf{X})(\mathbf{G}^{H}_{1}d\mathbf{G}_{1})^{-1}.
\end{eqnarray}
Since $\mathbf{A}=(\mathbf{I}_{m} - \mathbf{B})^{-1}-\mathbf{I}_{m}$, from (\ref{inv}), we have that
\begin{equation}\label{eq6}
    (d\mathbf{A}) = |\mathbf{I}_{m} - \mathbf{B}|^{-\beta(m-1)-2}(d\mathbf{B}).
\end{equation}
Substituting (\ref{eq5}) and (\ref{eq4}) into  (\ref{eq6}) we have
$$
  2^{m}|\mathbf{A}|^{-\beta(n-m+1)/2+1}(d\mathbf{Y})(\mathbf{H}^{H}_{1}d\mathbf{H}_{1})^{-1}
  \hspace{6cm}
$$
$$
  \hspace{4cm} = 2^{m}|\mathbf{I}_{m} -
  \mathbf{B}|^{-(m+1)}|\mathbf{B}|^{-\beta(n-m+1)/2+1}(d\mathbf{X})(\mathbf{G}^{H}_{1}d\mathbf{G}_{1})^{-1}.
$$
Then, by the uniqueness of the nonnormalised measure on Stiefel manifold,
$(\mathbf{H}^{H}_{1}d\mathbf{H}_{1}) = (\mathbf{G}^{H}_{1}d\mathbf{G}_{1})$. Thus,
$$
  (d\mathbf{Y}) = |\mathbf{A}|^{\beta(n-m+1)/2+1}|\mathbf{I}_{m} - \mathbf{B}|^{-\beta(m-1)-2}|\mathbf{B}|^{-\beta(n-m+1)/2+1}(d\mathbf{X}).
$$
And using $|\mathbf{A}| = |(\mathbf{I}_{m}-\mathbf{X}^{H}\mathbf{X})^{-1/2} \mathbf{X}^{H}\mathbf{X}
(\mathbf{I}_{m}-\mathbf{X}^{H}\mathbf{X})^{-1/2}| = |(\mathbf{I}_{m}-\mathbf{X}^{H}\mathbf{X})|^{-1}
|\mathbf{X}^{H}\mathbf{X}|$ and $|\mathbf{B}| = |\mathbf{X}^{H}\mathbf{X}|$,  the required result is obtained.

\textbf{ii}). The proof is similar to the preceding exposition given in \textbf{i}).
\end{proof}

\begin{corollary}\label{cor1}
Assume that $\mathbf{X}  \in {\mathcal L}_{m,n}^{\beta} $ and $\mathbf{Y}  \in {\mathcal L}_{m,n}^{\beta}$  are matrices of functionally independent variables.
\begin{description}
  \item[i)] Let $\mathbf{Y} = (1- \tr \mathbf{X}^{H}\mathbf{X})^{-1/2}\mathbf{X}$. Then
     \begin{equation}\label{eq21}
        (d\mathbf{Y}) = (1- \tr \mathbf{X}^{H}\mathbf{X})^{-(\beta nm/2+1)} (d\mathbf{X}).
     \end{equation}
  \item[ii)] If $\mathbf{X} = (1+ \tr\mathbf{Y}^{H}\mathbf{Y})^{-1/2}\mathbf{Y}$, we have
     \begin{equation}\label{eq31}
        (d\mathbf{X}) = (1+ \tr\mathbf{Y}^{H}\mathbf{Y})^{-(\beta nm/2+1)} (d\mathbf{Y}).
     \end{equation}
\end{description}
\end{corollary}
\begin{proof}
\textbf{i}). Observing that $\mathbf{Y} = (1- \tr \mathbf{X}^{H}\mathbf{X})^{-1/2}\mathbf{X}$ can be write as $\mathbf{y} =
(1-\mathbf{x}^{H}\mathbf{x})^{-1/2}\mathbf{x}$, where $\mathbf{x} = \vec \mathbf{X} \in \mathcal{L}_{1,nm}^{\beta}$ and
$\mathbf{y} = \vec \mathbf{Y} \in \mathcal{L}_{1,nm}^{\beta}$. Its proof is obtained as a particular case to that given 
for the theorem. \textbf{ii}). Its proof is analogous to one given to \textbf{i}).
\end{proof}

\begin{definition}\label{ellip}
It is said that the random matrix $\mathbf{Y} \in \mathcal{L}_{m,n}^{\beta}$ has a \emph{matrix
variate elliptical distribution}, denoted as $\mathbf{Y} \sim \mathcal{E}_{n \times
m}^{\beta}(\boldsymbol{\mu},\mathbf{\Theta},\mathbf{\Sigma}, h)$, if its density is
\begin{equation}\label{v-sd}
    \frac{1}{|\mathbf{\Sigma}|^{\beta n/2}|\mathbf{\Theta}|^{\beta m/2}}
    h\left \{ \beta\tr \left [\mathbf{\Sigma}^{-1} (\mathbf{Y} - \boldsymbol{\mu})^{H}
    \mathbf{\Theta}^{-1}(\mathbf{Y} - \boldsymbol{\mu})\right ]\right \}(d\mathbf{Y}).
\end{equation}  
where
\begin{equation}\label{cc}
    \int_{\mathfrak{P}_{1}^{\beta}} u^{mn\beta/2 -1}h(\beta u) du < \infty,
    \end{equation} 
and where $\mathbf{\Theta} \in
\mathfrak{P}_{n}^{\beta}$, $ \mathbf{\Sigma} \in \mathfrak{P}_{m}^{\beta}$ and
$\boldsymbol{\mu} \in \mathcal{L}_{m,n}^{\beta}$ are constant matrices.
\end{definition}

\texttt{}Finally, note that for $a \in \mathcal{L}_{1,1}^{\beta}$ constant, making the change of variable $v = u/a$ and $(du) =
a^{\beta}(dv)$ in \citet[Equation 2.21, p. 26]{fzn:90} we have
\begin{equation}\label{int}
    \int_{v \in \mathfrak{P}_{1}^{\beta}} v^{nm\beta/2-1}h(\beta a v) (dv)= \frac{a^{-nm\beta/2} \Gamma^{\beta}_{1}[nm\beta/2]}{\pi^{nm\beta/2}}
\end{equation}
Finally, from \citet{dggj:11a} 
\begin{equation}\label{int2}
  \int_{\mathfrak{P}^{\beta}_{m}}|\mathbf{V}|^{\beta(n - m + 1)/2 - 1} h\left(\beta \tr\mathbf{\Sigma}^{-1} \mathbf{V} \right)(d\mathbf{V})
   =\frac{|\mathbf{\Sigma}|^{\beta n/2}\Gamma^{\beta}_{m}[\beta n/2]}{\pi^{\beta nm/2}},
\end{equation}
where $\mathbf{V} \in \mathfrak{P}_{m}^{\beta}$ and $\mathbf{\Sigma} \in \mathfrak{P}_{m}^{\beta}$.

\section{Multimatrix variate and multimatricvariate distributions}\label{sec:3}

First, note that any new particular distribution indexed by the kernel $h(\cdot)$ form part of its density function, then it defines a family of distributions itself in terms of each possible choice of $h(\cdot)$. 

Assume that $\mathbf{X}\sim \mathcal{E}_{n \times m}^{\beta}(\mathbf{0}, \mathbf{I}_{n}, \mathbf{I}_{m}; h)$,
such that $n_{0}+n_{1}+ \cdots + n_{k} = n$, and $\mathbf{X} = \left(\mathbf{X}^{H}_{0},\mathbf{X}^{H}_{1},
\dots, \mathbf{X}^{H}_{k} \right)^{H}$. Then (\ref{v-sd}) can be written as
\begin{equation}\label{eq0}
    dF_{\mathbf{X}_{0},\mathbf{X}_{1}, \dots,\mathbf{X}_{k}}(\mathbf{X}_{0},\mathbf{X}_{1}, \dots,\mathbf{X}_{k}) =
     h[\beta\tr(\mathbf{X}^{H}_{0}\mathbf{X}_{0}+ \mathbf{X}^{H}_{1}\mathbf{X}_{1}+\cdots+\mathbf{X}^{H}_{k}
     \mathbf{X}_{k})]\bigwedge_{i=0}^{k}(d\mathbf{X}_{i}),
\end{equation}
where $\mathbf{X}_{i} \in \mathcal{L}_{m,n_{i}}^{\beta}$, $i = 0,1, \dots, k$.
Take in account that, only under a matrix variate normal distribution the random matrices are
independent, see \citet{fz:90}, \citet{gv:93} and \citet{fzn:90}. In general, the random matrices $\mathbf{X}_{0}, \mathbf{X}_{1}, \dots,\mathbf{X}_{k}$ are
probabilistically dependent.  

\begin{theorem}\label{mgge} 
Suppose that $\mathbf{X}\sim \mathcal{E}_{n \times m}^{\beta}(\mathbf{0}, \mathbf{I}_{n}, \mathbf{I}_{m}; h)$, with $\mathbf{X}_{i} \in
\mathcal{L}^{\beta}_{m, n_{i}}$, (remember that $n_{i} \geq m$), $i = 0,1, \dots, k$. 
\begin{description}
  \item[i)]  Define $V = \tr \mathbf{X}_{0}^{H} \mathbf{X}_{0}$. Then, the joint density $dF_{V,\mathbf{X}_{1}, \dots,\mathbf{X}_{k}}(v,\mathbf{X}_{1}, \dots,\mathbf{X}_{k})$ is given by
      \begin{equation}\label{mxgge}
          \frac{\pi^{n_{0}m\beta/2}}{\Gamma^{\beta}_{1}[n_{0}m\beta/2]} h\left[\beta\left(v+\displaystyle \tr\sum_{i=1}^{k} \mathbf{X}_{i}^{H} \mathbf{X}_{i}\right)\right]
          v^{n_{0}m\beta/2-1} (dv)\bigwedge_{i=1}^{k}\left(d\mathbf{X}_{i}\right),
      \end{equation}
      where $V \in \mathfrak{P}^{\beta}_{1}$. This distribution shall be termed \emph{multimatrix variate generalised Gamma - Elliptical distribution}.
  \item[ii)] Let $\mathbf{V} = \mathbf{X}_{0}^{H} \mathbf{X}_{0}$. Hence, the joint density $dF_{\mathbf{V},\mathbf{X}_{1}, \dots,\mathbf{X}_{k}}(\mathbf{V},\mathbf{X}_{1}, \dots,\mathbf{X}_{k})$ is given by
      \begin{equation}\label{mcgwe}
          \frac{\pi^{\beta n_{0}m/2}  |\mathbf{V}|^{\beta(n_{0} - m + 1)/2 - 1}}{\Gamma_{m}^{\beta}[\beta n_{0}/2]} h\left[\beta \tr\left(\mathbf{V}+\displaystyle \sum_{i=1}^{k} \mathbf{X}_{i}^{H} \mathbf{X}_{i}\right)\right]
          (d\mathbf{V})\bigwedge_{i=1}^{k}\left(d\mathbf{X}_{i}\right),
      \end{equation}
      where $\mathbf{V} \in \mathfrak{P}^{\beta}_{m}$. This distribution shall be called \emph{multimatricvariate generalised Whishart - Elliptical distribution}.
\end{description}
\end{theorem}
\begin{proof}
We have that the joint density function of $\mathbf{X}_{0},\mathbf{X}_{1}, \dots,\mathbf{X}_{k}$ is
\begin{equation}\label{ee}
     h[\beta\tr(\mathbf{X}^{H}_{0}\mathbf{X}_{0}+ \mathbf{X}^{H}_{1}\mathbf{X}_{1}+\cdots+\mathbf{X}^{H}_{k}
     \mathbf{X}_{k})]\bigwedge_{i=0}^{k}(d\mathbf{X}_{i}).
\end{equation}
\begin{description}
  \item[i)] Let $V = \tr \mathbf{X}_{0}^{H} \mathbf{X}_{0}$, hence by (\ref{w}) 
     $$
       (d\mathbf{X}_{0}) = 2^{-1}v^{n_{0}m \beta/2 - 1}(dv)\wedge(\mathbf{h}^{H}_{1}d\mathbf{h}_{1}),
     $$
  where $\mathbf{h}_{1} \in \mathcal{V}^{\beta}_{1,n_{0}m}$. Thus, the multimatrix variate joint density function  
  $$
  dF_{V, \mathbf{h}_{1},\mathbf{X}_{1}, \dots,\mathbf{X}_{k}}(v, \mathbf{h}_{1},\mathbf{X}_{1}, \dots,\mathbf{X}_{k})
  $$ 
  is
  \begin{equation*}
       \frac{v^{n_{0}m\beta/2 - 1}}{2} h\left[\beta\left(v+ \tr\sum_{i=1}^{k}\mathbf{X}^{H}_{i}\mathbf{X}_{i}\right )\right]
       (dv)\wedge (\mathbf{h}^{H}_{1}d\mathbf{h}_{1})\bigwedge_{i=1}^{k}(d\mathbf{X}_{i}).
  \end{equation*}
  By integration over $\mathbf{h}_{1} \in \mathcal{V}^{\beta}_{1,n_{0}m}$ using (\ref{vol}), the desired result is obtained.
  \item[ii)] Define $\mathbf{V} = \mathbf{X}_{0}^{H} \mathbf{X}_{0}$. Then by (\ref{w}), we have that 
  $$
    (d\mathbf{X}_{0}) = 2^{-m} |\mathbf{V}|^{\beta(n_{0} - m + 1)/2 - 1} (d\mathbf{V})(\mathbf{H}_{1}^{H}d\mathbf{H}_{1}),
  $$
  where $\mathbf{H}_{1} \in \mathcal{V}^{\beta}_{m,n_{0}}$. The desired result is obtained make de change of variable in (\ref{ee}) and integrating over $\mathbf{H}_{1} \in \mathcal{V}^{\beta}_{m,n_{0}}$, using (\ref{vol}).
\end{description} 
\end{proof}

Proceeding as \citet[Equation 4.2]{dgclpr:22}, defining $V_{i} = \tr\mathbf{X}_{i}^{H} \mathbf{X}_{i}$, $i = 0,1,\dots,k$, and as in \citet[Equation (1), p. 216]{dgcl:22} the following result in general case, is obtained.

\begin{theorem}\label{mgggw} Suppose that $\mathbf{X} = \left(\mathbf{X}^{H}_{0}, \dots,
\mathbf{X}^{H}_{k} \right)^{H}$ has  a matrix variate spherical distribution, with $\mathbf{X}_{i} \in
\mathcal{L}^{\beta}_{m, n_{i}}$, $i = 0,1, \dots, k$. Then: 
\begin{description}
  \item[i)] If we define $V_{i} = \tr\mathbf{X}_{i}^{H} \mathbf{X}_{i}$, $i = 0,1,\dots,k$, it is obtained
      \begin{equation}\label{mxgg}
        dF_{V_{0},\ldots,V_{k}}(v_{0},\ldots,v_{k})=\pi^{nm\beta/2}\prod_{i=0}^{k}\frac{v_{i}^{n_{i}m\beta/2-1}}
         {\Gamma_{1}^{\beta}\left[n_{i}m\beta/2\right]}
        h\left(\beta\sum_{i=0}^{k}v_{i}\right)\bigwedge_{i=0}^{k}(dv_{i}),
      \end{equation}
  where $n = n_{0}+\cdots+n_{k}$, $V_{i} \in \mathfrak{P}^{\beta}_{1}$, $i = 0,1,\dots,k$, which distribution shall be termed \emph{multivariate generalised Gamma  distribution}. 
  \item[ii)] The joint density $dF_{\mathbf{V}_{0},\mathbf{V}_{1}, \dots,\mathbf{V}_{k}}(\mathbf{V}_{0},\mathbf{V}_{1}, \dots,\mathbf{V}_{k})$ when $\mathbf{V}_{i} = \mathbf{X}_{i}^{H} \mathbf{X}_{i}$, $i = 0,1,\dots,k$.  is given by
      \begin{equation}\label{mcgw}
          \pi^{\beta nm/2} \prod_{i=0}^{k}\frac{|\mathbf{V}_{i}|^{\beta(n_{i} - m + 1)/2 - 1}}{\Gamma_{m}^{\beta}[\beta n_{i}/2]} h\left(\beta\displaystyle \tr\sum_{i=0}^{k} \mathbf{V}_{i}\right)
          \bigwedge_{i=0}^{k}\left(d\mathbf{V}_{i}\right),
      \end{equation}
      where $\mathbf{V}_{i} \in \mathfrak{P}^{\beta}_{m}$, $i = 0,1,\dots,k$. This distribution shall be termed \emph{multimatricvariate generalised Whishart distribution}.
\end{description}
\end{theorem}

Particular cases of these two distributions have been studied in the literature under a normal distribution in the real, complex and quaternionic cases, see \citet{n:07}, \citet{fz:90}, \citet{ln:82}, \citet{dgcl:22} and \citet{lx:09}, among many other authors. 

\begin{theorem}\label{ggpVII}
Assume that $\mathbf{X}\sim \mathcal{E}_{n \times m}^{\beta}(\mathbf{0}, \mathbf{I}_{n}, \mathbf{I}_{m}; h)$, with $\mathbf{X}_{i} \in \mathcal{L}^{\beta}_{m, n_{i}}$, $i = 0,1, \dots, k$. 
\begin{description}
  \item[i)] Define $V = \tr \mathbf{X}^{H}_{0}\mathbf{X}_{0}$ and $\mathbf{T}_{i} = V ^{-1/2}\mathbf{X}_{i}$, $i = 1,\dots,k$.\newline
  The joint density $dF_{V,\mathbf{T}_{1}, \dots,\mathbf{T}_{k}}(v,\mathbf{T}_{1}, \dots,\mathbf{T}_{k})$
  is given by
\begin{equation}\label{mxggp7}
   \frac{\pi^{n_{0}m\beta/2}}{\Gamma_{1}^{\beta}[n_{0}m\beta/2]} h\left[\beta v\left(1+\displaystyle\sum_{i=1}^{k}\tr \mathbf{T}^{H}_{i}\mathbf{T}_{i}\right)\right]
    v^{nm\beta/2-1} (dv)\bigwedge_{i=1}^{k}\left(d\mathbf{T}_{i}\right),
\end{equation}
  where $n = n_{0}+n_{1}+\cdots+n_{k}$, $V  \in \mathfrak{P}^{\beta}_{1}$ and $\mathbf{T}_{i} \in \mathcal{L}^{\beta}_{m, n_{i}}$, $i =
  1,\dots,k$. This distribution shall be termed \emph{multimatrix variate generalised Gamma - Pearson type VII
  distribution}.
  \item[ii)] Define $\mathbf{V} = \mathbf{X}^{H}_{0}\mathbf{X}_{0}$ and $\mathbf{T}_{i} = \mathbf{X}_{i} \mathbf{V}^{-1/2}$, $i =
  1,\dots,k$.\\ Then, the joint density   $dF_{\mathbf{V},\mathbf{T}_{1},\dots,\mathbf{T}_{k}}(\mathbf{V},\mathbf{T}_{1}, \dots,\mathbf{T}_{k})$ is given by
\begin{equation}\label{mcgwp7}
   \frac{\pi^{\beta n_{0}m/2}}{\Gamma^{\beta}_{m}[\beta n_{0}/2]}|\mathbf{V}|^{\beta(n - m + 1)/2 - 1}
  h\left[\beta \tr \mathbf{V}\left(\mathbf{I}_{m}+\displaystyle\sum_{i=1}^{k}\mathbf{T}^{H}_{i}\mathbf{T}_{i}\right)\right]
  (d\mathbf{V})\bigwedge_{i=1}^{k}\left(d\mathbf{T}_{i}\right),
\end{equation}
where $n = n_{0}+n_{1}+\cdots+n_{k}$, $\mathbf{V}  \in \mathfrak{P}^{\beta}_{m}$ and $\mathbf{T}_{i} \in \mathcal{L}^{\beta}_{m, n_{i}}$, $i = 1,\dots,k$. This distribution shall be called \emph{multimatricvariate generalised Wishart-T distribution}.
\end{description}
\end{theorem}
\begin{proof}
\begin{description}
  \item[i)] The density (\ref{mxggp7}) is follow from (\ref{mxgge}) defining $\mathbf{T}_{i}= V^{-1/2}\mathbf{X}_{i}$, $i = 1,\dots,k$. Hence by Proposition \ref{lem1}, we have
     $$
       \bigwedge_{i=1}^{k}(d\mathbf{X}_{i})= v^{(n-n_{0})m\beta/2}\bigwedge_{i=1}^{k}(d\mathbf{T}_{i}).
     $$
    Then the required result follows.
  \item[ii)] Now, from (\ref{mcgwe}), making the change of variable $\mathbf{T}_{i}= \mathbf{X}_{i}\mathbf{V}^{-1/2}$, $i = 1,\dots,k$ and considering that
      $$
        \bigwedge_{i=1}^{k}(d\mathbf{X}_{i})= |\mathbf{V}|^{\beta (n-n_{0})m/2}\bigwedge_{i=1}^{k}(d\mathbf{T}_{i}).
      $$
      by Proposition \ref{lem1}. The result is obtained. 
\end{description}
\end{proof}

\begin{corollary}\label{mp7}
Under the hypotheses of the theorem
\begin{description}
  \item[i)]The marginal density $dF_{\mathbf{T}_{1}, \dots,\mathbf{T}_{k}}(\mathbf{T}_{1}, \dots,\mathbf{T}_{k})$ is termed  \emph{multimatrix variate Pearson type VII} and is given by
    \begin{equation}\label{mxp7}
        \frac{\Gamma^{\beta}_{1}[nm\beta/2]}{\pi^{(n-n_{0})m\beta/2}\Gamma^{\beta}_{1}[n_{0}m\beta/2]}
        \left(1+\tr\displaystyle\sum_{i=1}^{k}\mathbf{T}^{H}_{i}\mathbf{T}_{i}\right)^{-nm\beta/2}
        \bigwedge_{i=1}^{k}\left(d\mathbf{T}_{i}\right),
    \end{equation}
    where $\mathbf{T}_{i} \in \mathcal{L}^{\beta}_{m, n_{i}}$ and $n = n_{0}+n_{1}+\cdots+n_{k}$.
  \item[ii)] Similarly, the termed  \emph{multimatricvariate Pearson VII distribution} is the marginal density
     $dF_{\mathbf{T}_{1}, \dots,\mathbf{T}_{k}}(\mathbf{T}_{1}, \dots,\mathbf{T}_{k})$ of $\mathbf{T}_{1},
     \dots,\mathbf{T}_{k}$ and is given by
     \begin{equation}\label{mcp7}
       \frac{\Gamma^{\beta}_{m}[\beta n/2]}{\pi^{\beta (n-n_{0})m/2}\Gamma^{\beta}_{m}[\beta n_{0}/2]}
       \left|\mathbf{I}_{m}+\displaystyle\sum_{i=1}^{k}\mathbf{T}^{H}_{i}\mathbf{T}_{i}\right|^{-\beta n/2}
      \bigwedge_{i=1}^{k}\left(d\mathbf{T}_{i}\right).
     \end{equation}
      With $\mathbf{T}_{i} \in \mathcal{L}^{\beta}_{m, n_{i}}$ and $n = n_{0}+n_{1}+\cdots+n_{k}$.
\end{description}
\end{corollary}
\begin{proof}
\begin{description}
  \item[i)] Integrating (\ref{mxggp7}) over  $V  \in \mathfrak{P}^{\beta}_{1}$ using (\ref{int}) the density (\ref{mxp7}) is obtained.
  \item[ii)] Analogously, integrating (\ref{mcgwp7}) over $\mathbf{V}  \in \mathfrak{P}^{\beta}_{m}$, using (\ref{int2}), is obtained 
    $$
      \int_{\mathfrak{P}^{\beta}_{m}}|\mathbf{V}|^{\beta(n - m + 1)/2 - 1} h\left[\beta \tr \mathbf{V} \left(\mathbf{I}_{m}+\displaystyle\sum_{i=1}^{k}\mathbf{T}^{H}_{i}\mathbf{T}_{i}\right)\right](d\mathbf{V})
      \hspace{3cm} 
    $$
    $$
      \hspace{5cm}
      =\frac{\Gamma^{\beta}_{m}[\beta n/2]\left|\mathbf{I}_{m}+\displaystyle\sum_{i=1}^{k}\mathbf{T}^{H}_{i}\mathbf{T}_{i} 
      \right|^{-\beta n/2}}{\pi^{\beta nm/2}},
    $$
  and the desired result is archived.
\end{description}
\end{proof}

\begin{theorem}\label{mggwpII}
Assume that $\mathbf{X} = \left(\mathbf{X}^{H}_{0}, \dots,\mathbf{X}^{H}_{k} \right)^{H}$ has  a matrix variate
spherical distribution, with $\mathbf{X}_{i} \in \mathcal{L}^{\beta}_{m, n_{i}}$, $i = 0,1, \dots, k$. 
\begin{description}
  \item[i)]  Define $V = \tr \mathbf{X}^{H}_{0}\mathbf{X}_{0}$ and $\mathbf{R}_{i} = \left(V + \tr \mathbf{X}^{H}_{i}\mathbf{X}_{i}\right)^{-1/2} \mathbf{X}_{i}$, $i = 1,\dots,k$. Then the joint density of $V,\mathbf{R}_{1}, \dots,\mathbf{R}_{k}$, denoted as $dF_{V,\mathbf{R}_{1},\dots,\mathbf{R}_{k}}(v,\mathbf{R}_{1}, \dots,\mathbf{R}_{k})$, is given by
      $$
        \frac{\pi^{n_{0}m\beta/2}}{\Gamma^{\beta}_{1}[n_{0}m\beta/2]} v^{nm\beta/2-1} h\left[\beta v\left( 1+\displaystyle\sum_{i=1}^{k}
        \frac{ \tr \mathbf{R}^{H}_{i}\mathbf{R}_{i}}{\left(1-  \tr \mathbf{R}^{H}_{i}\mathbf{R}_{i}\right)}\right)\right ]\hspace{3cm}
      $$
      \begin{equation}\label{mxggp2}
         \hspace{2cm}
         \times  \prod_{i=i}^{k}\left(1-  \tr \mathbf{R}^{H}_{i}\mathbf{R}_{i}\right)^{-n_{i}m\beta/2-1}
         (dv)\bigwedge_{i=1}^{k}\left(d\mathbf{R}_{i}\right),
      \end{equation}
   where $n = n_{0}+n_{1}+\cdots+n_{k}$, $V  \in \mathfrak{P}^{\beta}_{1}$ and $\mathbf{R}_{i} \in \mathcal{L}^{\beta}_{m, n_{i}}$, and $\tr \mathbf{R}^{H}_{i} \mathbf{R}_{i} < 1$ $i = 1,\dots,k$. This distribution shall be termed \emph{multimatrix variate generalised Gamma-Pearson type II distribution}.
  \item[ii)]  Define $\mathbf{V} = \mathbf{X}^{H}_{0}\mathbf{X}_{0}$ and $\mathbf{R}_{i} = \mathbf{X}_{i}(\mathbf{V}+\mathbf{X}_{i}^{H}\mathbf{X}_{i})^{-1/2}$, $i = 1,\dots,k$. Then the joint density of $\mathbf{V},\mathbf{R}_{1}, \dots,\mathbf{R}_{k}$, denoted as $dF_{\mathbf{V},\mathbf{R}_{1}, \dots,\mathbf{R}_{k}}(\mathbf{V},\mathbf{R}_{1},\dots,\mathbf{R}_{k})$, can be written as
      $$
         \frac{\pi^{\beta n_{0}m/2}}{\Gamma^{\beta}_{m}[\beta n_{0}/2]}|\mathbf{V}|^{\beta(n - m + 1)/2 - 1} h\left[\beta\tr \mathbf{V}\left(\mathbf{I}_{m}+\displaystyle\sum_{i=1}^{k}(\mathbf{I}_{m} - \mathbf{R}^{H}_{i} \mathbf{R}_{i})^{-1}\mathbf{R}^{H}_{i}\mathbf{R}_{i}\right)\right]
      $$
      \begin{equation}\label{mcgwp2}\hspace{2cm}
           \times  \prod_{i=i}^{k}|\mathbf{I}_{m}- \mathbf{R}^{H}_{i}\mathbf{R}_{i}|^{\beta(n_{i} - m + 1)/2 - 1}
           (d\mathbf{V})\bigwedge_{i=1}^{k}\left(d\mathbf{R}_{i}\right),
      \end{equation}
  where $n = n_{0}+n_{1}+\cdots+n_{k}$, $\mathbf{V}\in \mathfrak{P}^{\beta}_{m}$ and $\mathbf{R}_{i} \in \mathcal{L}^{\beta}_{m, n_{i}}$,  $ \mathbf{I}_{m}- \mathbf{R}^{H}_{i}\mathbf{R}_{i} \in \mathfrak{P}^{\beta}_{m}$, $i = 1,\dots,k$. This distribution shall be termed \emph{multimatricvariate generaliosed Wishart-Pearson type II distribution}. 
\end{description}
\end{theorem}
\begin{proof}
\begin{description}
   \item[i)] Consider in (\ref{mxggp7}) the change of variable $\mathbf{T}_{i}=(1-tr \mathbf{R}^{H}_{i}\mathbf{R}_{i})^{-1/2}\mathbf{R}_{i}$, $i=1,\dots,k$, hence by Theorem \ref{cor1},
       $$
         \bigwedge_{i=1}^{k}\left(d\mathbf{T}_{i}\right) = \prod_{i=i}^{k}\left(1- tr \mathbf{R}^{H}_{i}\mathbf{R}_{i}\right)^{-(\beta n_{i}m/2+1)} \bigwedge_{i=1}^{k}\left(d\mathbf{R}_{i}\right),
       $$
       and the desired result follows.
   \item[ii)] Define $\mathbf{T}_{i} = \mathbf{R}_{i} (\mathbf{I}_{m}-\mathbf{X}_{i}^{H}\mathbf{X}_{i})^{-1/2}$, $i=1,\dots,k$, with
     $$
        \bigwedge_{i=1}^{k}\left(d\mathbf{T}_{i}\right) = \left|\mathbf{I}_{m} - \mathbf{R}^{H}\mathbf{R}\right|^{-\beta(n_{i}+m+1)/2-1} \bigwedge_{i=1}^{k}\left(d\mathbf{R}_{i}\right).
     $$
     The result is follow making this change of variable in (\ref{mcgwp7}). 
\end{description} 
\end{proof}

Similarly to Corollary \ref{mp7}, from (\ref{mxggp2}) and (\ref{mcgwp2}) we can obtain the multimatrix variate and multimatricvariate marginal densities $dF_{\mathbf{R}_{1}, \dots,\mathbf{R}_{k}}$.

\begin{corollary}\label{cor2} 
Consider the assumptions of the Theorem \ref{mggwpII}. Then 
\begin{description}
  \item[i)] The multimatrix variate marginal density $dF_{\mathbf{R}_{1}, \dots,\mathbf{R}_{k}}(\mathbf{R}_{1},
\dots,\mathbf{R}_{k})$ is
$$
  \frac{\Gamma^{\beta}_{1}[nm\beta/2]}{\pi^{(n-n_{0})m\beta/2}\Gamma^{\beta}_{1}[n_{0}m\beta/2]} \left[ 1+\displaystyle\sum_{i=1}^{k}
  \frac{\tr\mathbf{R}_{i}^{H}\mathbf{R}_{i}}{\left(1- \tr\mathbf{R}_{i}^{H}\mathbf{R}_{i}\right)}\right ]^{-nm\beta/2} \hspace{2cm}
$$
\begin{equation}\label{mxp2}
  \hspace{3cm}\times  \prod_{i=i}^{k}\left(1-
  \tr\mathbf{R}_{i}^{H}\mathbf{R}_{i}\right)^{-n_{i}m\beta/2-1}
  \bigwedge_{i=1}^{k}\left(d\mathbf{R}_{i}\right),
\end{equation}
which shall  be termed \emph{multimatrix variate Pearson type II distribution}.
  \item[ii)] The multimaticvariate marginal density $dF_{\mathbf{R}_{1}, \dots,\mathbf{R}_{k}}(\mathbf{R}_{1},
\dots,\mathbf{R}_{k})$ is
$$
   \frac{\Gamma^{\beta}_{m}[\beta n/2]}{\pi^{\beta(n-n_{0})m/2}\Gamma^{\beta}_{m}[n_{0}/2]}
  \left|\mathbf{I}_{m}-\displaystyle\sum_{i=1}^{k} (\mathbf{I}_{m}-\mathbf{R}^{H}_{i}\mathbf{R}_{i})^{-1}\mathbf{R}^{H}_{i}\mathbf{R}_{i}\right|^{-\beta n/2}
  \hspace{5cm}
$$
\begin{equation}\label{mcp2}
  \hspace{2cm}
  \times \prod_{i=1}^{k}\left|\mathbf{I}_{m}-\mathbf{R}^{H}_{i}\mathbf{R}_{i}\right|^{\beta(n_{i}-m+1)/2-1}
  \bigwedge_{i=1}^{k}\left(d\mathbf{R}_{i}\right),
\end{equation}
which shall be termed \emph{multimatricvariate Pearson type II distribution}.
\end{description}
Where $n = n_{0}+n_{1}+\cdots+n_{k}$, and $\mathbf{R}_{i} \in \mathcal{L}^{\beta}_{m, n_{i}}$, $i = 1,\dots,k$.
\end{corollary}
\begin{proof}
  \begin{description}
    \item[i)] First, (\ref{mxp2}) is archived integrating (\ref{mxggp2}) over $V \in \mathfrak{P}^{\beta}_{1}$ using (\ref{int}).
    \item[ii)] Similarly, (\ref{mcp2}) is follows integrating (\ref{mcgwp2}) over $V \in \mathfrak{P}^{\beta}_{m}$ using (\ref{int2}).
  \end{description}
\end{proof}

\begin{theorem}\label{mggbII}
Assuming the hypotheses of Theorem \ref{ggpVII} and defining $\mathbf{F}_{i} =
\mathbf{T}^{H}_{i}\mathbf{T}_{i} > \mathbf{0}$, $i = 1, \dots,k$. 
\begin{description}
  \item[i)] Then the joint density $dF_{V, \mathbf{F}_{1}, \dots,\mathbf{F}_{k}}(v,\mathbf{F}_{1}, \dots,\mathbf{F}_{k})$ is
      $$
        \frac{\pi^{nm\beta/2}v^{nm\beta/2-1}}{\Gamma^{\beta}_{1}[n_{0}m\beta/2]}\prod_{i=1}^{k} \left(\frac{|\mathbf{F}_{i}|^{\beta(n_{i}-m+1)/2-1}}{\Gamma^{\beta}_{m}[n_{i}\beta/2]}\right )
        \hspace{3cm}
      $$
      \begin{equation}\label{mxggb2}
          \hspace{3cm} \times
          h\left[\beta v\left(1+\displaystyle\sum_{i=1}^{k}\tr\mathbf{F}_{i}\right)\right]
          (dv)\bigwedge_{i=1}^{k}\left(d\mathbf{F}_{i}\right).
      \end{equation}
      This distribution shall be termed \emph{multimatrix variate generalised Gamma-beta type II distribution}.
  \item[ii)] Then the joint density $dF_{\mathbf{V}_{0}, \mathbf{F}_{1},\dots,\mathbf{F}_{k}}(\mathbf{V},\mathbf{F}_{1}, \dots,\mathbf{F}_{k})$ is
      $$
        \frac{\pi^{\beta m n/2}}{\displaystyle\prod_{i=0}^{k}\Gamma^{\beta}_{m}[\beta n_{i}/2]}
        |\mathbf{V}|^{\beta(n-m+1)/2-1} \prod_{i=1}^{k}|\mathbf{F}_{i}|^{\beta(n_{i}-m+1)/2-1}
        \hspace{3cm}
      $$      
      \begin{equation}\label{mcgwb2}
         \hspace{2cm}
         \times h\left(\beta\tr \mathbf{V}\left(\mathbf{I}_{m}+\displaystyle\sum_{i=1}^{k}\mathbf{F}_{i}\right)\right) (d\mathbf{V})\bigwedge_{i=1}^{k}\left(d\mathbf{F}_{i}\right).
      \end{equation}
      This distribution can be termed \emph{multimatricvariate generalised Wishart-beta type II distribution}.  
\end{description}
\end{theorem}
\begin{proof}
Multimatrix variate and multimatricvariate density functions (\ref{mxggb2}) and (\ref{mcgwb2}) are obtained considering the change of variable $\mathbf{F}_{i} = \mathbf{T}^{H}_{i}\mathbf{T}_{i}$, $i =1,\dots,k$ in expressions (\ref{mxggp7}) and (\ref{mcgwp7}), respectively. Observing that by (\ref{lemW})
$$
  \bigwedge_{i=1}^{k}\left(d\mathbf{T}_{i}\right) = 2^{-mk}  \prod_{i=1}^{k}|\mathbf{F}_{i}|^{\beta(n_{i}-m+1)/2-1} \bigwedge_{i=1}^{k}\left(d\mathbf{F}_{i}\right)\bigwedge_{i=1}^{k}\left(\mathbf{H}^{H}_{1_{i}}d\mathbf{H}_{1_{i}}\right)
$$ 
where $\mathbf{H}_{1_{i}}\in \mathcal{V}^{\beta}_{n_{i},m}$, $i=1,\dots,k$. Hence, integrating over $\mathbf{H}_{1_{i}}\in \mathcal{V}^{\beta}_{n_{i},m}$ $i=1,\dots,k$ by (\ref{vol}) we have
$$
\int_{\mathbf{H}_{1_{1}}} \cdots \int_{\mathbf{H}_{1_{k}}} \bigwedge_{i=1}^{k}\left(\mathbf{H}^{H}_{1_{i}}d\mathbf{H}_{1_{i}}\right) = \frac{2^{mk} \pi^{\beta(n-n_{0})m/2}}{\displaystyle\prod_{i=1}^{k} \Gamma^{\beta}_{m}[\beta n_{i}/2]}
$$
and the proof is complete. 
\end{proof}

\begin{corollary}\label{mxmcbII}
The corresponding marginal densities  $dF_{\mathbf{F}_{1}, \dots, \mathbf{F}_{k}}(\mathbf{F}_{1}, \dots,\mathbf{F}_{k})$ of (\ref{mxggb2}) is
  \begin{description}
    \item[i)] 
        \begin{equation}\label{mxb2}
           \frac{\Gamma^{\beta}_{1}[nm\beta/2]}{\Gamma^{\beta}_{1}[n_{0}m\beta/2]} \prod_{i=1}^{k}\left(
           \frac{|\mathbf{F}_{i}|^{\beta(n_{i}-m+1)/2-1}}{\Gamma^{\beta}_{m}[n_{i}\beta/2]}\right) \left(1+\displaystyle\sum_{i=1}^{k} \tr\mathbf{F}_{i}\right)^{-nm\beta/2} \bigwedge_{i=1}^{k}\left(d\mathbf{F}_{i}\right).
        \end{equation}
        Whose distribution can be termed \emph{multimatrix variate beta type II distribution}.
    \item[ii)] And associated marginal density function $dF_{\mathbf{F}_{1}, \dots, \mathbf{F}_{k}}(\mathbf{F}_{1}, \dots,\mathbf{F}_{k})$ of (\ref{mcgwb2}) is given by
        \begin{equation}\label{mcb2}
           \frac{\Gamma^{\beta}_{m}[n/2]}{\displaystyle\prod_{i=0}^{k}\Gamma^{\beta}_{m}[\beta n_{i}/2]}
           \prod_{i=1}^{k}|\mathbf{F}_{i}|^{\beta(n_{i}-m+1)/2-1}
           \left|\mathbf{I}_{m}+\displaystyle\sum_{i=1}^{k}\mathbf{F}_{i}\right|^{-\beta n/2}
           \bigwedge_{i=1}^{k}\left(d\mathbf{F}_{i}\right).
        \end{equation}
    This distribution shall be termed \emph{multimatricvariate beta type II distribution}.
  \end{description}
\end{corollary}
\begin{proof}
  The proof of both result are obtained integrating (\ref{mxggb2}) and (\ref{mcgwb2}) with respect to $V \in \mathfrak{P}^{\beta}_{1}$ and $\mathbf{V} \in \mathfrak{P}^{\beta}_{m}$ using (\ref{int}) and (\ref{int2}), respectively. Or proceeding exactly as in the proof of Theorem \ref{mggbII} 
\end{proof}

The density function (\ref{mxb2}) contain as particular case to the distribution in real case, proposed in \citet[Problem 3.18, p.118]{mh:05}.

\begin{theorem}\label{mggwbI}
Assuming that $\mathbf{B}_{i} = \mathbf{R}^{H}_{i}\mathbf{R}_{i} \in \mathfrak{P_{m}^{\beta}}$, $i = 1, \dots,k$, in Theorem \ref{mggwpII}.
\begin{description}
  \item[i)]  Then the joint density $dF_{V, \mathbf{B}_{1},\dots,\mathbf{B}_{k}}(v, \mathbf{B}_{1},
      \dots,\mathbf{B}_{k})$, with $\tr \mathbf{B}_{i} < 1$, $i = 1, \dots,k$ is
      $$
        \frac{\pi^{nm\beta/2} v^{nm\beta/2-1}}{\Gamma^{\beta}_{1}[n_{0}m\beta/2]}
        \prod_{i=1}^{k}\left(\frac{|\mathbf{B}_{i}|^{\beta(ni-m+1)/2-1}}{\Gamma^{\beta}_{m}[n_{i}\beta/2]}\right )
        h\left[\beta v\left( 1+\displaystyle\sum_{i=1}^{k} \frac{\tr\mathbf{B}_{i}}{\left(1-
        \tr\mathbf{B}_{i}\right)}\right)\right ]\hspace{3cm}
      $$
      \begin{equation}\label{mxggb1}
         \hspace{4cm}
         \times  \prod_{i=i}^{k}\left(1- \tr \mathbf{B}_{i}\right)^{-n_{i}m\beta/2-1}
         (dv)\bigwedge_{i=1}^{k}\left(d\mathbf{B}_{i}\right).
      \end{equation}
      This distribution shall be termed \emph{multimatrix variate generalised Gamma-beta type I distribution}.
  \item[ii)] Then the joint density $dF_{\mathbf{V},\mathbf{B}_{1},\dots,\mathbf{B}_{k}}(\mathbf{V},
     \mathbf{B}_{1}, \dots,\mathbf{B}_{k})$, where $\mathbf{I}_{m}- \mathbf{B}_{i} \in \mathfrak{P}^{\beta}_{m} $, $i =1,2,\dots, k$ is
     $$
       \frac{\pi^{\beta m n/2}}{\displaystyle\prod_{i=0}^{k}\Gamma^{\beta}_{m}[\beta n_{i}/2]}
      |\mathbf{V}|^{\beta(n-m+1)/2-1} \prod_{i=1}^{k}\left(\frac{|\mathbf{B}_{i}|}{|\mathbf{I}_{m} -
      \mathbf{B}_{i}|}\right)^{\beta(n_{i}-m+1)/2-1} \hspace{2cm}
     $$
    \begin{equation}\label{mcgwb1}
      \hspace{2cm}
      \times h\left(\beta\tr \mathbf{V}\left(\mathbf{I}_{m}+\displaystyle\sum_{i=1}^{k}(\mathbf{I}_{m} -
      \mathbf{B}_{i})^{-1}\mathbf{B}_{i}\right)\right) (d\mathbf{V})\bigwedge_{i=1}^{k}\left(d\mathbf{B}_{i}\right).
    \end{equation}
    This distribution can be termed \emph{multimatricvariate generalised Wishart-beta type I distribution}.  
\end{description}
\end{theorem}
\begin{proof}
  Making the change of variable $\mathbf{B}_{i} = \mathbf{R}^{H}_{i}\mathbf{R}_{i}$, $i = 1, \dots,k$ in the densities (\ref{mxggp2}) and (\ref{mcgwp2}), respectively, and noting that 
  $$
    \bigwedge_{i=1}^{k}\left(d\mathbf{R}_{i}\right) = 2^{-mk}  \prod_{i=1}^{k}|\mathbf{B}_{i}|^{\beta(n_{i}-m+1)/2-1} \bigwedge_{i=1}^{k}\left(d\mathbf{B}_{i}\right)\bigwedge_{i=1}^{k}\left(\mathbf{H}^{H}_{1_{i}}d\mathbf{H}_{1_{i}}\right)
  $$ 
  where $\mathbf{H}_{1_{i}}\in \mathcal{V}^{\beta}_{n_{i},m}$, $i=1,\dots,k$. The proof is archived, integrating over $\mathbf{H}_{1_{i}}\in \mathcal{V}^{\beta}_{n_{i},m}$ $i=1,\dots,k$ using (\ref{vol}). In which case
  $$
    \int_{\mathbf{H}_{1_{1}}} \cdots \int_{\mathbf{H}_{1_{k}}} \bigwedge_{i=1}^{k}\left(\mathbf{H}^{H}_{1_{i}}d\mathbf{H}_{1_{i}}\right) = \frac{2^{mk} \pi^{\beta(n-n_{0})m/2}}{\displaystyle\prod_{i=1}^{k} \Gamma^{\beta}_{m}[\beta n_{i}/2]}.
  $$
\end{proof}

Integrating (\ref{mxggb1}) and (\ref{mcgwb1}) with respect to $v$ and $\mathbf{V}$ using (\ref{int}) and (\ref{int2}), respectively; we obtain the marginal densities $dF_{\mathbf{B}_{1}, \dots, \mathbf{B}_{k}}(\mathbf{B}_{1}, \dots,\mathbf{B}_{k})$ of multimatrix variate and multimatricvariate beta type I distribution. Summarising:

\begin{corollary}\label{mbI}
  \begin{description}
    \item[i)] The density function $dF_{\mathbf{B}_{1}, \dots, \mathbf{B}_{k}}(\mathbf{B}_{1}, \dots,\mathbf{B}_{k})$ can written as
        $$
          \frac{\Gamma^{\beta}_{1}[nm\beta/2]}{\Gamma^{\beta}_{1}[n_{0}m\beta/2]} \prod_{i=1}^{k}\left(\frac{|\mathbf{B}_{i}|^{\beta(n_{i}-m+1)/2-1}}{\Gamma^{\beta}_{m}[n_{i}\beta/2]\left(1-
           \tr \mathbf{B}_{i}\right)^{n_{i}m\beta/2+1}}\right)     \hspace{3cm}
        $$
        \begin{equation}\label{mxb1}
           \hspace{4cm}
           \times  \left( 1+\displaystyle\sum_{i=1}^{k} \frac{\tr\mathbf{B}_{i}}{\left(1-
          \tr\mathbf{B}_{i}\right)}\right )^{-nm\beta/2} \bigwedge_{i=1}^{k}\left(d\mathbf{B}_{i}\right).
        \end{equation}
        Whose marginal distribution shall be termed \emph{multimatrix variate beta type I distribution}.
    \item[ii)] In this case, the density function $dF_{\mathbf{B}_{1}, \dots, \mathbf{B}_{k}}(\mathbf{B}_{1}, \dots,\mathbf{B}_{k})$
       is 
       $$
         \frac{\Gamma^{\beta}_{m}[\beta n/2]}{\displaystyle\prod_{i=0}^{k}\Gamma^{\beta}_{m}[\beta n_{i}/2]}
         \prod_{i=1}^{k}\left(\frac{|\mathbf{B}_{i}|}{|\mathbf{I}_{m} -
        \mathbf{B}_{i}|}\right)^{\beta(n_{i}-m+1)/2-1} \hspace{4cm}
       $$
    \begin{equation}\label{mcb1}
       \hspace{2cm}
       \times \left|\mathbf{I}_{m}+\displaystyle\sum_{i=1}^{k}(\mathbf{I}_{m} -
       \mathbf{B}_{i})^{-1}\mathbf{B}_{i}\right|^{-\beta n/2} \bigwedge_{i=1}^{k}\left(d\mathbf{B}_{i}\right).
    \end{equation}
    This marginal distribution shall be named \emph{multimatricvariate beta type I distribution}.
  \end{description}
\end{corollary}

Finally

\begin{theorem}\label{ggW} Assume that $\mathbf{X} = \left(\mathbf{X}^{H}_{0}, \dots,
\mathbf{X}^{H}_{k} \right)^{H}$ has  a matrix variate spherical distribution, with $\mathbf{X}_{i} \in
\mathcal{L}^{\beta}_{m,n_{i}}$, $n_{i} \geq m$, $i = 0,1, \dots, k$. Define $V = \tr\mathbf{X}^{H}_{0}\mathbf{X}_{0}$ and
$\mathbf{V}_{i} = \mathbf{X}^{H}_{i}\mathbf{X}_{i}$, $i = 1, \dots,k$. \newline Then, the joint density
$dF_{V,\mathbf{V}_{1}, \dots,\mathbf{V}_{k}}(v,\mathbf{V}_{1}, \dots,\mathbf{V}_{k})$ is given by
$$
   \frac{\pi^{\beta nm/2} v^{\beta nm/2-1}}{\Gamma^{\beta}_{1}[\beta n_{0}m/2]} \prod_{i=1}^{k}\left(\frac{|\mathbf{V}_{i}|^{\beta(ni-m+1)/2-1}}{\Gamma^{\beta}_{m}[\beta n_{i}/2]}\right)
   \hspace{5cm}
$$
\begin{equation}\label{mggw}
   \hspace{3cm}
   h\left[\beta\left(v+\displaystyle\sum_{i=1}^{k}\tr \mathbf{V}_{i}\right)\right]
   (dv)\bigwedge_{i=1}^{k}\left(d\mathbf{V}_{i}\right),
\end{equation}
where $V \in \mathfrak{P}^{\beta}_{1}$, $\mathbf{V}_{i} \in \mathfrak{P}^{\beta}_{m}$, $i = 1, \dots,k$. This distribution shall be named \emph{multimatrix variate generalised Gamma - generalised Wishart distribution}.
\end{theorem}
\begin{proof}
From Theorem \ref{mgge} we have
 \begin{equation*}
   \frac{\pi^{n_{0}m\beta/2} v^{n_{0}m\beta/2-1}}{\Gamma^{\beta}_{1}[n_{0}m\beta/2]} h\left[\beta\left(v+\tr\displaystyle\sum_{i=1}^{k}\mathbf{X}^{H}_{i}\mathbf{X}_{i}\right)\right]
    (dv)\bigwedge_{i=1}^{k}\left(d\mathbf{X}_{i}\right).
\end{equation*} 
Defining $\mathbf{V}_{i} = \mathbf{X}^{H}_{i}\mathbf{X}_{i}$ with $i = 1, \dots,k$ and proceeding as in the proof of Theorem \ref{mggwbI}, the result is immediate.
\end{proof}

\section{Some properties and extensions}\label{sec:4}

Now, multimatrix and multimatric variate distributions for two and three matrix arguments can be achived. Then we can obtain two or more different classes of marginal distributions.  The metodology can be extended to more than three different marginal distributions. In addition, the inverse distributions of some  multimatrix and multimatric variate distributions  are also obtained.

\begin{theorem}\label{mggpVIIpII} 
Let be $\mathbf{X} = \left(\mathbf{X}^{H}_{0},\mathbf{X}^{H}_{1},\mathbf{X}^{H}_{2}
\right)^{H}$ has  a matrix variate spherical distribution, with $\mathbf{X}_{i} \in \mathcal{L}^{\beta}_{m,n_{i}}$. 
\begin{description}
  \item[i)] Define $V_{0} = \tr \mathbf{X}_{0}^{H}\mathbf{X}_{0}$, $\mathbf{T} = V_{0}^{-1/2}\mathbf{X}_{1}$, and $\mathbf{R} = \mathbf{R} = V^{-1/2}\mathbf{X}_{2}$, where $V =(V_{0}+\tr \mathbf{X}_{2}^{H}\mathbf{X}_{2})$. The joint density $dF_{V,\mathbf{T},\mathbf{R}}(v,\mathbf{T}, \mathbf{R})$ is given by
      $$
        \frac{\pi^{n_{0}m\beta/2}v^{nm\beta/2-1}}{\Gamma\mathbf{\beta}_{1}[n_{0}m\beta/2]}
        h\left\{\beta v \left[1+\left(1-\tr\mathbf{R}^{H}\mathbf{R}\right)\tr\mathbf{T}^{H}\mathbf{T}\right]\right\}\hspace{5cm}
      $$
      \begin{equation}\label{mxggp7p2}
         \hspace{3cm}\times \left(1-\tr\mathbf{R}^{H}\mathbf{R}\right)^{(n_{0}+n_{1})m\beta/2-1}(dv)\wedge (d\mathbf{T})\wedge d\mathbf{R}).
      \end{equation}
      where $n = n_{0}+n_{1}+ n_{2}$, $V \in \mathfrak{P}^{\beta}_{1}$, $\mathbf{T}\in \mathcal{L}^{\beta}_{m,n_{1}}$, $\mathbf{R}\in
      \mathcal{L}^{\beta}_{m,n_{2}}$ such that $\tr\mathbf{R}^{H}\mathbf{R} \leq 1$. This distribution shall be termed \emph{threematrix variate generalised Gamma - Pearson type VII - Pearson type II distribution}. 
  \item[ii)] Let be $\mathbf{V}_{0} = \mathbf{X}^{H}_{0}\mathbf{X}_{0}$, $\mathbf{T} = \mathbf{X}_{1}\mathbf{V}_{0}^{-1/2}$, $\mathbf{V} = \mathbf{V}_{0}+\mathbf{X}^{H}_{2}\mathbf{X}_{2}$, and $\mathbf{R}= \mathbf{X}_{2}\mathbf{V}^{-1/2}$. Then the joint density $dF_{\mathbf{V},\mathbf{T}, \mathbf{R}}(\mathbf{V}, \mathbf{T},\mathbf{R})$ is given by
      $$
        \frac{\pi^{\beta n_{0}m/2}}{\Gamma^{\beta}_{m}[\beta n_{0}/2]}|\mathbf{V}|^{\beta(n-m+1)/2-1}
        h\left\{\beta\tr \left[\mathbf{V}+ (\mathbf{I}_{m}- \mathbf{R}^{H}\mathbf{R})\mathbf{V}^{1/2} \mathbf{T}^{H}\mathbf{TV}^{1/2}\right]\right\}
      $$
      \begin{equation}\label{mcgwp7p2}
          \hspace{2cm}\times  |\mathbf{I}_{m}- \mathbf{R}^{H}\mathbf{R}|^{\beta(n_{0}+n_{1}-m+1)/2-1}
          (d\mathbf{V})\wedge(d\mathbf{T})\wedge \left(d\mathbf{R}\right),
      \end{equation}
      where $n = n_{0}+n_{1}+n_{2}$, $\mathbf{V} \in \mathfrak{P}^{\beta}_{m}$, $\mathbf{I}_{m}- \mathbf{R}^{H}\mathbf{R} \in \mathfrak{P}^{\beta}_{m}$, $\mathbf{T} \in \mathcal{L}^{\beta}_{m,n_{1}}$ and $\mathbf{R} \in  \mathcal{L}^{\beta}_{m,n_{1}}$. The distribution of $\mathbf{V},\mathbf{T},\mathbf{R}$ shall be termed \emph{trimatricvariate generalised Wishart-Pearson VII-Pearson type II distribution}. 
\end{description}
\end{theorem}
\begin{proof}
\begin{description}
  \item[i)] From Theorem \ref{mgge}, $dF_{V_{0},\mathbf{X}_{1},\mathbf{X}_{2}}(v_{0},\mathbf{X}_{1},\mathbf{X}_{2})$ is
     \begin{equation*}
        \frac{\pi^{n_{0}m\beta/2}}{\Gamma^{\beta}_{1}[n_{0}m\beta/2]} h\left[\beta \left(v_{0}+\tr \mathbf{X}_{1}^{H}\mathbf{X}_{1} 
        + \tr \mathbf{X}_{2}^{H}\mathbf{X}_{2}\right )\right]
        v_{0}^{n_{0}m\beta/2-1} \hspace{2cm} 
     \end{equation*} 
     \begin{equation}\label{aux0}\hspace{5cm}
       (dv_{0})\wedge\left(d\mathbf{X}_{1}\right)\wedge\left(d\mathbf{X}_{2}\right).
     \end{equation}
     Let $V= V_{0} + \tr \mathbf{X}_{2}^{H}\mathbf{X}_{2}$, $\mathbf{T} = V_{0}^{-1/2}\mathbf{X}_{1}$, and $\mathbf{R} = V^{-1/2}\mathbf{X}_{2}$. Thus, $\mathbf{X}_{1} = V_{0}^{1/2}\mathbf{T}$,  $\mathbf{X}_{2} = V^{1/2}\mathbf{R}$, and  
     $$
       V_{0}= V - \tr \mathbf{X}_{2}^{H}\mathbf{X}_{2} = V - V\tr\mathbf{R}^{H}\mathbf{R} = V(1-\tr\mathbf{R}^{H}\mathbf{R}). 
     $$
     Thus $\mathbf{T} =  [V(1-\tr\mathbf{R}^{H} \mathbf{R})]^{-1/2}\mathbf{X}_{1}$ and $(dv)=(dv_{0})$. Then, the volume element $(dv_{0})\wedge\left(d\mathbf{X}_{1}\right)\wedge\left(d\mathbf{X}_{2}\right)$ is
    
     \begin{equation}\label{aux3}
       v^{(n_{1}+n_{2})m\beta/2} (1-\tr\mathbf{R}^{H} \mathbf{R})^{n_{1}m\beta/2} (dv)\wedge\left(d\mathbf{T}\right)\wedge\left(d\mathbf{R}\right).
     \end{equation}
     From (\ref{aux0}), substituting $V= V_{0} + \tr \mathbf{X}_{2}^{H}\mathbf{X}_{2}$,  $\mathbf{X}_{1} =  [V(1-\tr\mathbf{R}^{H} \mathbf{R})]^{-1/2}\mathbf{T}$ and by (\ref{aux3}), the desired result is obtained.
  \item[ii)] From (\ref{mcgwe}), $dF_{\mathbf{V}_{0},\mathbf{X}_{1},\mathbf{X}_{2}}(\mathbf{V}_{0},\mathbf{X}_{1},\mathbf{X}_{2})$ is
     $$
       \frac{\pi^{\beta n_{0}m/2}  |\mathbf{V}_{0}|^{\beta(n_{0} - m + 1)/2 - 1}}{\Gamma_{m}^{\beta}[\beta n_{0}/2]}
       \hspace{7cm}
     $$
     $$
       \hspace{1cm}
       \times h\left[\beta \tr\left(\mathbf{V}_{0}+\mathbf{X}_{1}^{H} \mathbf{X}_{1}+\mathbf{X}_{2}^{H} \mathbf{X}_{2}\right)\right]
       (d\mathbf{V}_{0})\wedge\left(d\mathbf{X}_{1}\right)\wedge \left(d\mathbf{X}_{2}\right).
     $$
     Define $\mathbf{V}= \mathbf{V}_{0} + \mathbf{X}_{2}^{H}\mathbf{X}_{2}$, $\mathbf{T} = \mathbf{X}_{1}\mathbf{V}_{0}^{-1/2}$, and $\mathbf{R} = \mathbf{X}_{2}\mathbf{V}^{-1/2}$. Hence by Proposition \ref{lem1}, 
     $$
       (d\mathbf{V}_{0})\wedge (d\mathbf{X}_{1})\wedge(d\mathbf{X}_{2}) = |\mathbf{V}_{0}|^{\beta n_{1}/2}
       |\mathbf{V}_{0}+ \mathbf{X}_{2}^{H}\mathbf{X}_{2}|^{\beta n_{2}/2}(d\mathbf{V})\wedge
        (d\mathbf{T})\wedge(d\mathbf{R}).
     $$  
     But, $\mathbf{X}_{1} =  \mathbf{T}\mathbf{V}_{0}^{1/2}$,  $\mathbf{X}_{2} = \mathbf{R}\mathbf{V}^{1/2}$, and  
     $$
       \mathbf{V}_{0}= \mathbf{V} - \mathbf{X}_{2}^{H}\mathbf{X}_{2} = \mathbf{V} - \mathbf{V}^{1/2}\mathbf{R}^{H}\mathbf{R}\mathbf{V}^{1/2} = \mathbf{V}^{1/2}(\mathbf{I}_{m}-\mathbf{R}^{H}\mathbf{R})\mathbf{V}^{1/2}. 
     $$  
     This way, $\mathbf{X}_{1} = \mathbf{T}\left(\mathbf{V}^{1/2}(\mathbf{I}_{m}-\mathbf{R}^{H}\mathbf{R})\mathbf{V}^{1/2}\right)^{1/2}$ and $(d\mathbf{V})=(d\mathbf{V}_{0})$. Therefore
     $$
       (d\mathbf{V}_{0})\wedge\left(d\mathbf{X}_{1}\right)\wedge\left(d\mathbf{X}_{2}\right) = |\mathbf{V}|^{\beta(n_{1}+n_{2})/2} |\mathbf{I}_{m}-\mathbf{R}^{H} \mathbf{R}|^{\beta n_{1}/2} (d\mathbf{V})\wedge\left(d\mathbf{T}\right)\wedge\left(d\mathbf{R}\right)
     $$
    Given that
    $$
    |\mathbf{V}_{0}+ \mathbf{X}_{2}^{H}\mathbf{X}_{2}| = |\mathbf{V}| \mbox{ and } |\mathbf{V}_{0}| =
    |\mathbf{V}||\mathbf{I}_{m}- \mathbf{R}^{H}\mathbf{R}|.
  $$
  Finally, observe that $\tr \mathbf{X}^{H}_{1}\mathbf{X}_{1} = \tr \mathbf{V}^{1/2}(\mathbf{I}_{m} - \mathbf{R}^{H}\mathbf{R})\mathbf{V}^{1/2} \mathbf{T}^{H}\mathbf{T}$. Then, the required result is obtained. 
\end{description} 
\end{proof}

\begin{corollary} Under the Hypotheses of Theorem \ref{mggpVIIpII}, define $\mathbf{F} = \mathbf{T}^{H}\mathbf{T}$ and $\mathbf{B} = \mathbf{R}^{H}\mathbf{R}$. 
  \begin{description}
    \item[i)] The density function of the termed \emph{trimatrix variate generalised Gamma - Pearson type VII - Pearson type II
       distribution}, $dF_{\mathbf{T},\mathbf{R}}(\mathbf{T}, \mathbf{R})$  is given by
       $$
         \frac{\pi^{nm\beta/2}v^{nm\beta/2-1} |\mathbf{F}|^{\beta(n_{1}-m+1)/2-1} |\mathbf{B}|^{\beta(n_{2}-m+1)/2-1}}{\Gamma^{\beta}_{1}[n_{0}m\beta/2]
         \Gamma^{\beta}_{m}[n_{1}\beta/2] \Gamma^{\beta}_{m}[n_{2}m\beta/2]} \left(1-\tr\mathbf{B}\right)^{(n_{0}+n_{1})m\beta/2-1}
       $$
      \begin{equation}\label{mxggb2b1}
         \times  h\left\{\beta v \left[1+\left(1-\tr\mathbf{B}\right) \tr\mathbf{F}\right]\right\} (dv)\wedge d\mathbf{F})\wedge d\mathbf{B}).
      \end{equation}
    \item[ii)] Similarly, the density $dF_{\mathbf{V},\mathbf{F}, \mathbf{U}}(\mathbf{V}, \mathbf{F}, \mathbf{U})$ is
       $$
          \frac{\pi^{\beta nm/2}|\mathbf{V}|^{\beta(n-m+1)/2-1} }{\Gamma^{\beta}_{m}[\beta n_{0}/2]\Gamma^{\beta}_{m}[\beta n_{1}/2]\Gamma^{\beta}_{m}[\beta n_{2}/2]} h\left[\beta\tr \left(\mathbf{V}+ (\mathbf{I}_{m}- \mathbf{B})\mathbf{V}^{1/2} \mathbf{FV}^{1/2}\right)\right] \hspace{2cm}
       $$
       \begin{equation*}
          \hspace{2cm}
          \times  |\mathbf{I}_{m}- \mathbf{B}|^{\beta(n_{0}+n_{1}-m+1)/2-1} |\mathbf{F}|^{\beta(n_{1}-m1)/2-1} |\mathbf{B}|^{\beta(n_{2}-m+1)/2-1}  
       \end{equation*}   
          \begin{equation}\label{mcgwb2b1}
          \hspace{8cm}
          (d\mathbf{V})\wedge(d\mathbf{F})\wedge \left(d\mathbf{B}\right),
       \end{equation}
       This density shall be termed \emph{trimatricvariate generalised Wishart-beta type II-beta type I distribution}.
  \end{description}   
     Where $n = n_{0} + n_{1} +n_{2}$,$V \in \mathfrak{P}^{\beta}_{1}$, $\mathbf{F} \in \mathfrak{P}^{\beta}_{m}$, $\mathbf{B} \in \mathfrak{P}^{\beta}_{m}$, $\mathbf{V} \in \mathfrak{P}^{\beta}_{m}$, $\tr\mathbf{B} \leq 1$ and $\mathbf{I}_{m}-\mathbf{B} \in \mathfrak{P}^{\beta}_{m}$.
\end{corollary}
\begin{proof}
The proofs of (\ref{mxggb2b1}) and (\ref{mcgwb2b1}) are follows from  (\ref{mxggp7p2}) and (\ref{mcgwp7p2}), respectively; making the change of variables $\mathbf{F} = \mathbf{T}^{H}\mathbf{T}$, $\mathbf{B}= \mathbf{R}^{H}\mathbf{R}$, and by Proposition \ref{lemW},
$$
  (d\mathbf{T})\wedge d\mathbf{R}) = 2^{-2m}|\mathbf{F}|^{\beta(n_{1}-m+1)/2-1} |\mathbf{B}|^{\beta(n_{2}-m+1)/2-1} (d\mathbf{F})\wedge d\mathbf{B}) \bigwedge_{i=1}^{2}\left(\mathbf{H}^{H}_{1_{i}}d\mathbf{H}_{1_{i}}\right).
$$
The desired results are archived, integrating over $\mathbf{H}_{1_{i}}\in \mathcal{V}^{\beta}_{n_{i},m}$ $i=1,2$ using (\ref{vol}). Moreover
  $$
    \int_{\mathbf{H}_{1_{1}}} \int_{\mathbf{H}_{1_{2}}} \bigwedge_{i=1}^{2}\left(\mathbf{H}^{H}_{1_{i}}d\mathbf{H}_{1_{i}}\right) = \frac{2^{2m} \pi^{\beta(n-n_{0})m/2}}{\displaystyle\prod_{i=1}^{2} \Gamma^{\beta}_{m}[\beta n_{i}/2]}.
  $$
\end{proof}

Additionally, we are interested in the distributions of the inverse of one or more of the arguments in the multimatrix variate or multimatricvariate distributions, which shall be termed inverse multimatrix variate or inverse multimatricvariate distributions.

\begin{theorem}\label{mibI}
  Define $\mathbf{A}_{i} = \mathbf{B}^{-1}_{i}$, $i = 1, \dots,r$.
  \begin{description}
    \item[i)] Assume that $\mathbf{B}_{1}, \cdots, \mathbf{B}_{r},\mathbf{B}_{r+1},\cdots,  \mathbf{B}_{k}$ have a multimatrix variate beta type I distribution. The join density function 
        $$
          dF_{\mathbf{A}_{1}, \cdots, \mathbf{A}_{r},\mathbf{B}_{r+1},\cdots,  \mathbf{B}_{k}}(\mathbf{A}_{1}, \cdots, \mathbf{A}_{r},\mathbf{B}_{r+1},\cdots, \mathbf{B}_{k}),
        $$
        is
        $$
          \frac{\Gamma^{\beta}_{1}[nm\beta/2]}{\Gamma^{\beta}_{1}[n_{0}m\beta/2] \prod_{i=1}^{k}\Gamma^{\beta}_{m}[n_{i}\beta/2]} \prod_{i=1}^{r}\left(\frac{|\mathbf{A}_{i}|^{-\beta(n_{i}+m-1)/2-1}}{\left(1-
           \tr \mathbf{A}^{-1}_{i}\right)^{n_{i}m\beta/2+1}}\right) \hspace{2cm}  
        $$
        $$
          \hspace{2cm}
          \times  \left( 1+\displaystyle\sum_{i=1}^{r} \frac{\tr\mathbf{A}^{-1}_{i}}{\left(1-
          \tr\mathbf{A}^{-1}_{i}\right)} + \sum_{i=r+1}^{k} \frac{\tr\mathbf{B}_{i}}{\left(1-
          \tr\mathbf{B}_{i}\right)}\right )^{-nm\beta/2}
        $$
        \begin{equation}\label{mxib1}
           \hspace{2cm}
           \times \prod_{i=r+1}^{k}\left(\frac{|\mathbf{B}_{i}|^{\beta(n_{i}-m+1)/2-1}}{\left(1 -
           \tr \mathbf{B}_{i}\right)^{n_{i}m\beta/2+1}}\right)
           \bigwedge_{i=1}^{r}\left(d\mathbf{A}_{i}\right) \bigwedge_{i=r+1}^{k}\left(d\mathbf{B}_{i}\right),
        \end{equation}
        where $\mathbf{A}_{i} \in \mathfrak{P}^{\beta}_{m}$, $\mathbf{B}_{i} \in \mathfrak{P}^{\beta}_{m}$, $\tr \mathbf{A}_{i} < 1$,  $\tr\mathbf{B}_{i} < 1$. 
    \item[ii)] Consider that $\mathbf{B}_{1}, \cdots, \mathbf{B}_{r},\mathbf{B}_{r+1},\cdots,  \mathbf{B}_{k}$ have a multimatricvariate beta type I distribution. Then, the join density function  
        $$
          dF_{\mathbf{A}_{1}, \cdots, \mathbf{A}_{r},\mathbf{B}_{r+1},\cdots,  \mathbf{B}_{k}}(\mathbf{A}_{1}, \cdots, \mathbf{A}_{r},\mathbf{B}_{r+1},\cdots, \mathbf{B}_{k})
        $$ is given by
        $$
          \frac{\Gamma^{\beta}_{m}[\beta n/2]}{\displaystyle\prod_{i=0}^{k}\Gamma^{\beta}_{m}[\beta n_{i}/2]}
          \prod_{i=1}^{r}\left(\frac{|\mathbf{A}_{i}|^{-\beta(m-1)-2}}{|\mathbf{A}_{i} - \mathbf{I}_{m}        |^{\beta(n_{i}-m+1)/2-1}}\right) \hspace{5cm}
        $$
        $$
        \hspace{2cm}
          \times \left|\mathbf{I}_{m}+\sum_{i=1}^{r}(\mathbf{A}_{i} - \mathbf{I}_{m})^{-1} + \sum_{i=r+1}^{k}(\mathbf{I}_{m} -
         \mathbf{B}_{i})^{-1}\mathbf{B}_{i}\right|^{-\beta n/2} 
        $$
        \begin{equation}\label{mcib1}
          \hspace{1.3cm}
          \times \prod_{i=r+1}^{k}\left(\frac{|\mathbf{B}_{i}|}{|\mathbf{I}_{m} -
          \mathbf{B}_{i}|}\right)^{\beta(n_{i}-m+1)/2-1}
          \bigwedge_{i=1}^{r}\left(d\mathbf{A}_{i}\right) \bigwedge_{i=r+1}^{k}\left(d\mathbf{B}_{i}\right),
        \end{equation}
    where $\mathbf{A}_{i} \in \mathfrak{P}^{\beta}_{m}$, $\mathbf{B}_{i} \in \mathfrak{P}^{\beta}_{m}$, $\mathbf{A}_{i} - \mathbf{I}_{m} \in \mathfrak{P}^{\beta}_{m}$,  $\mathbf{I}_{m}-\mathbf{B}_{i} \in \mathfrak{P}^{\beta}_{m}$
  \end{description}
\end{theorem} 
\begin{proof}
The density functions (\ref{mxib1}) and (\ref{mcib1}) are archived from (\ref{mxb1}) and (\ref{mcb1}), respectively, defining $\mathbf{A}_{i} = \mathbf{B_{i}}^{-1}$, $i = 1, \cdots,r$, and using the Proposition \ref{lemi}. 
\end{proof}

The parameter domain of the real normed division algebra in the multimatrix and multimatric variate distributions can be extended as in the real and complex cases. However, the statistical and/or geometrical interpretation, perhaps can be lost. In any case, these distributions are valid if we replace $n_{i}/2$ by  $a_{i}$, $n_{0}m/2$ by $a_{0}$ and $nm/2$ by $a$. Where the $a'^{s}$ are complex number with positive real part. From practical point of view, this parameter domain extension its allow to use nonlinear optimisation rather integer nonlinear optimisation in the procedure of estimation, among other possibilities.

Each distribution can be reparametrised in order to obtain a general expression for its density function. As in the normal case, the expressions obtained in this article appear in their standard form. 

If $(\mathbf{V}_{0},\mathbf{V}_{1}, \dots,\mathbf{V}_{k})$ follows a \emph{multimatricvariate or multimatrix variate generalised Whishart distribution}, with density function (\ref{mcgw}). Define $\mathbf{W}_{i} = \mathbf{\Sigma}_{i}^{1/2}\mathbf{V}_{i}\mathbf{\Sigma}_{i}^{1/2}$, $\mathbf{\Sigma}_{i} \in \mathfrak{P}^{\beta}_{m}$, $i = 0,1,\dots,k$, then by Proposition \ref{lems}, we have that
$$
  \bigwedge_{i=0}^{k}\left(d\mathbf{V}_{i}\right) = \prod_{i=0}^{k}|\mathbf{\Sigma}_{i}|^{-\beta(m-1)/2-1} \bigwedge_{i=0}^{k}\left(d\mathbf{W}_{i}\right).
$$ 
Then the density $dF_{\mathbf{W}_{0},\mathbf{W}_{1}, \dots,\mathbf{W}_{k}}(\mathbf{W}_{0},\mathbf{W}_{1}, \dots,\mathbf{W}_{k})$ is
$$
  \pi^{\beta nm/2} \prod_{i=0}^{k}\left (\frac{|\mathbf{\Sigma}_{i}^{-1/2}\mathbf{W}_{i}\mathbf{\Sigma}_{i}^{-1/2}|^{\beta(n_{i} - m + 1)/2 - 1}}{\Gamma_{m}^{\beta}[\beta n_{i}/2]}\right) h\left(\beta\displaystyle \tr\sum_{i=0}^{k} \mathbf{\Sigma}_{i}^{-1/2}\mathbf{W}_{i}\mathbf{\Sigma}_{i}^{-1/2}\right)
$$
$$
  \hspace{8cm}
  \prod_{i=0}^{k}|\mathbf{\Sigma}_{i}|^{-\beta(m-1)/2-1} \bigwedge_{i=0}^{k}\left(d\mathbf{W}_{i}\right).
$$
Hence
$$
  \pi^{\beta nm/2} \prod_{i=0}^{k}\left (\frac{|\mathbf{W}_{i}|^{\beta(n_{i} - m + 1)/2 - 1}}{\Gamma_{m}^{\beta}[\beta n_{i}/2]|\mathbf{\Sigma}_{i}|^{\beta n_{i}/2}}\right) h\left(\beta\displaystyle \tr\sum_{i=0}^{k} \mathbf{\Sigma}_{i}^{-1}\mathbf{W}_{i}\right) \bigwedge_{i=0}^{k}\left(d\mathbf{W}_{i}\right),
$$
where $\mathbf{W}_{i} \in \mathfrak{P}^{\beta}_{m}$, $i = 0,1,\dots,k$. 

\section{Example}\label{sec:5}

In this Section we provide an example in quaternions. Finding a  random suitable data base for this algebra is difficult, then we try first to explain a way of generating a number of applications by using  data bases of the literature of shape theory. 

We start with a known representation of a quaternion number in terms of $2\times 2$ complex matrices. Let $q=a+b\mathbf{i}+c\mathbf{j}+d\mathbf{k}$ be a quaternion, then $q$ can be written in terms of a the following $2\times 2$ matrix of complex entries:
\begin{equation}\label{z}
    \mathbf{Z}=
\begin{bmatrix}
a+b\mathbf{i} & c+d\mathbf{i} \\
-c+d\mathbf{i} & a-b\mathbf{i} 
\end{bmatrix}
\end{equation}

Thus $\mathbf{Z}$ can be seen as an array of 4 complex points, with a double symmetry: $a+b\mathbf{i}$ and $a-b\mathbf{i}$ are symmetric respect the $\Re$ axis; meanwhile $c+d\mathbf{i}$ and $-c+d\mathbf{i}$ are symmetric about the $\Im$ axis.

Now, shape theory deals, for example, with sets of planar figures summarised by corresponding landmarks between populations in order to obtain means, variability and discrimination via statistics on certain quotients spaces. Instead of the noisy Euclidean space, the statistics is performed with equivalent classes after filtering out some non meaning geometrical information, such as scaling, translation, rotation, reflection, etc.. Thus, a landmark data with the referred symmetries can be set into a vector variate quaternion sample for estimation of the extended real parameters $a_{i}=n_{i}/2, i=1,\ldots,k$ of the distributions here derived.  We focus on the mouse vertebra landmark data given for example in \citet{dm:98}. The sample consists of 23 small, 23 large and 30 control second thoracic vertebrae with 60 landmarks. After transforming the data for a suitable application of the complex matrix representation, we have in figure \ref{fig1} an example of an small vertebra.
\begin{figure}[htp]
    \centering
    \includegraphics[width=4cm]{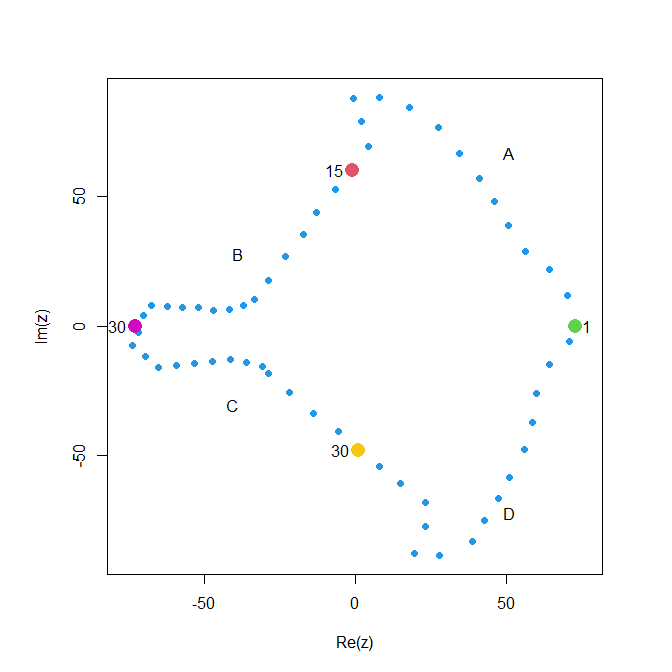}
    \caption{A 60 landmark small vertebra with high symmetry, divided in four parts $A, B, C, D$ for suitable application of a quaternionic complex matrix representation.}
    \label{fig1}
\end{figure}
The required sample most follows the symmetric suggestions of $\mathbf{Z}$. For getting this end, just cut the bones on landmarks $30$ and $45$ and two sectors are obtained: $ABC$, from landmarks $1$ to $45$; and,  $D$, from landmarks $46$ to $60$. Now the free part $D$ can be placed symmetrically to sector $B$ (landmarks $15$ to $30$) as a reflection on the imaginary axis $\Im$. Finally, for each bone we have 14 pairs of landmarks $(a_{u}+b_{u}\mathbf{i},c_{u+14}+d_{u+14}\mathbf{i})$, each one representing the quaternion $q_{u}=a_{u}+b_{u}\mathbf{i}+c_{u+14}\mathbf{j}+d_{u+14}\mathbf{k}$, where $u=2,\ldots,15$. Namely, for each $u=2,\ldots,15$, the first landmark $a_{u}+b_{u}\mathbf{i}$ belongs to the sector $A$ (symmetric to sector $D$, respect $\Re$ axis) and it is paired with the second landmark $c_{u+15}+d_{u+15}\mathbf{i}$ (symmetric to the translated sector $C$). Summarising, the sample for the three classes of bones consists of upper landmarks 2 to 29, distributed by the pairs $(2,16), (3,17), \cdots, (14,28), (15,29)$. With each pair providing a quaternion, we have the following three dependent samples: 23 quaternion vectors of size 14 for the small group (Figure \ref{fig2}), 23 quaternion vectors of size 14 for the large class (Figure \ref{fig3}), and 30 quaternion vectors of size 14 for the control set (Figure \ref{fig4}).

\begin{figure}
     \centering
     \begin{subfigure}[b]{0.3\textwidth}
         \centering
         \includegraphics[width=\textwidth]{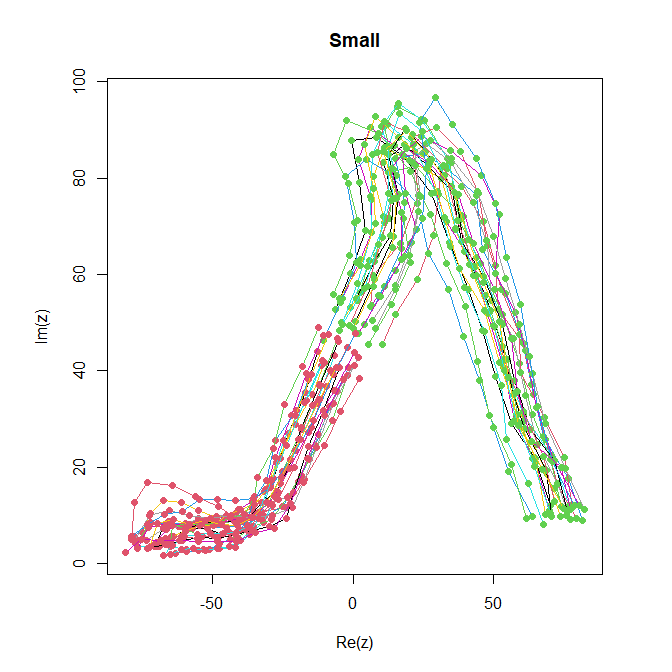}
         \caption{}
         \label{fig2}
     \end{subfigure}
     \hfill
          \centering
     \begin{subfigure}[b]{0.3\textwidth}
         \centering
         \includegraphics[width=\textwidth]{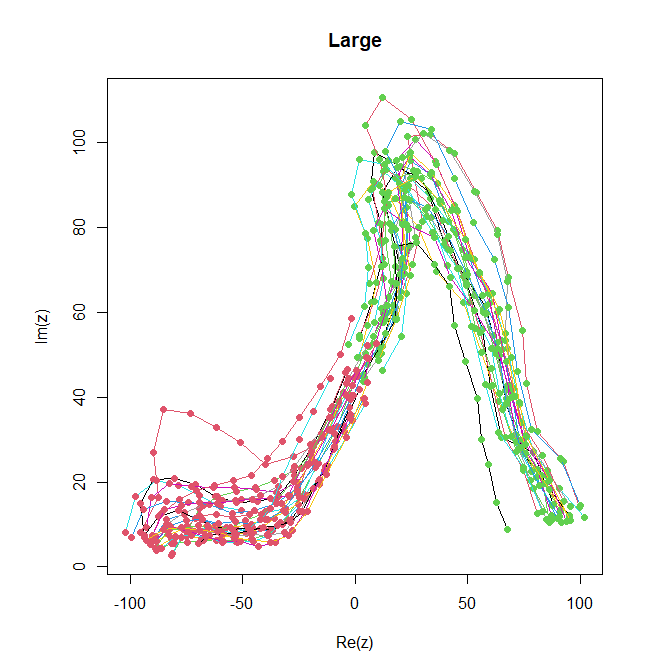}
         \caption{}
         \label{fig3}
     \end{subfigure}
     \hfill
     \centering
     \begin{subfigure}[b]{0.3\textwidth}
         \centering
         \includegraphics[width=\textwidth]{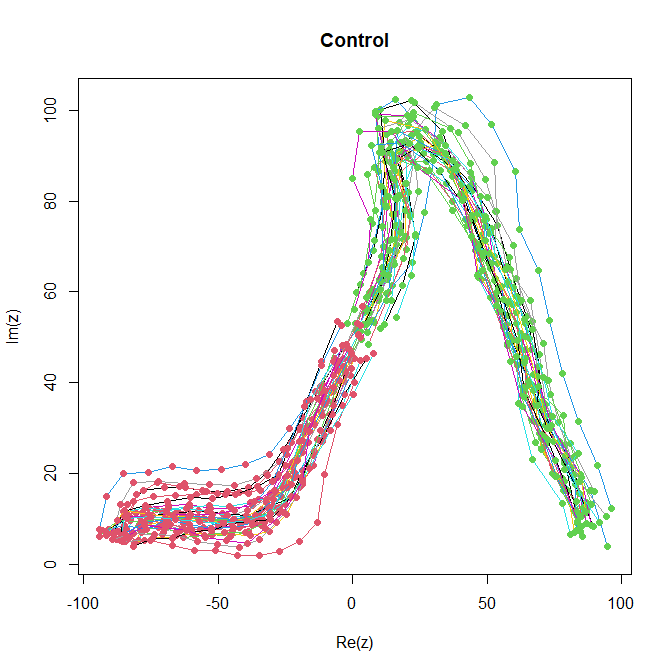}
         \caption{}
         \label{fig4}
     \end{subfigure}
\caption{\small{(a) Dependent sample of 23 quaternion vectors of size 14 for the small group. (b) Dependent sample of 23 quaternion vectors of size 14 for the large class. (c) Dependent sample of 30 quaternion vectors of size 14 for the control set. The sector $A$ (green) with complex $a_{u}+b_{u}\mathbf{i}$ constitutes the first two components of the quaternion $q_{u}=a_{u}+b_{u}\mathbf{i}+c_{u+14}\mathbf{j}+d_{u+14}\mathbf{k}$, and the sector $B$ (red) with complex $c_{u+15}+d_{u+15}\mathbf{i}$ indexes the last two components of $q_{u}, u=2,\ldots,15$.}}
\end{figure}

The mouse vertebra landmark data has been studied in several works based on the classical assumptions of normality and independent probabilistic sample, see for example \citet{dm:98} and the references therein. Gaussian restriction can not deal properly with the outlier shapes, meanwhile assuming a sample of independent small and large bones just facilitates the estimations via likelihood function, but it seems to be out of underlying sample extraction and population description given in the original source of the experiments back to the earlies 70s. 

This work provides two solutions of the previous problems. First, within the multimatricvariate and multimatrix variate distributions, the choice of distributions that are invariant under the family of $\beta-$elliptically contoured distributions. A notorious advantage facing the lack of knowledge about the random matrix law for the landmark data. And second, the distributions here derived give join density functions of dependent matrices, a fact eliminates the historical controversy between the probabilistic independence of the sample data. Another important fact, hidden for the classical studies based on Gaussianity (independently equivalent) resides on the real normed division algebra supporting the landmark data. In particular, the use of quaternions via the complex matrix representation integrates in the study the high symmetry of the bones. Finally, for this data, the distributions under quaternions are extremely simple, they are reduced to vectors instead of the real matrix setting given in the referred studies.

For the sake of illustration and simple computation, we consider the multimatrix variate beta type II distribution (\ref{mxb2}) as the likelihood function dependent sample invariant under the quaternionic elliptically contoured distribution.

For the small and large samples, $k=23$, $\beta=4$ and $m=1$, meanwhile in the control group  $k=30$. In the three samples we use (\ref{mxb2}) for the maximum likelihood estimates of $a_{0}=n_{0}/2$ and $a=n_{i}/2, i=1,\ldots,k$. Here, $\mathbf{T}_{i}$ is a $14\times 1$ quaternion vector, then $F_{i}=\mathbf{T}_{i}^{H}\mathbf{T}_{i}$ is a real value for all $i=1,\ldots,k$. Thus the likelihood in (\ref{mxb2}) takes the form:
$$
\frac{\Gamma_{1}^{\beta}\left[(a_{0}+k a)m\beta\right]}{\Gamma_{1}^{\beta}\left[a_{0}m\beta\right]\left(\Gamma_{m}^{\beta}\left[a\beta\right]\right)^{k}}
\prod_{i=1}^{k}F_{i}^{\beta\left(a-\frac{m}{2}+\frac{1}{2}\right)-1}\left(1+\sum_{i=1}^{k}F_{i}\right)^{-(a_{0}+ka)m\beta}.
$$
The computations are performed in the Optimx package of R under several methods of optimisations and a wide range of seeds for a consistent estimation. The results are given in the following table

\begin{center}
\begin{tabular}{||c || c c ||} 
 \hline
 Sample & $\hat{a}_{0}$ & $\hat{a}$  \\ [0.5ex] 
 \hline\hline
 Small & 0.040714 & 45.194923 \\ 
 \hline
 Large & 0.03294941 & 12.82131179\\
 \hline
 Control & 0.03765324 & 32.99296063\\
 \hline
 \hline
\end{tabular}
\end{center}

Finally, we can use again the symmetry of the modified mouse vertebra data in order to define $14\times 2$ quaternion matrices $\mathbf{T}_{i}, i=1,\ldots,k$. In this case the first column is the same formed by sectors $A$ and $B$, and the second column corresponds to sector $D$ (the reflection of sector $A$) and the translated and reflected sector $C$, which is symmetric respect sector $B$. Then we obtain the $2\times 2$ quaternion matrix $\mathbf{F}_{i} = \mathbf{T}^{H}_{i}\mathbf{T}_{i}$, $i =1,\dots,k$, and the likelihood function (\ref{mxb2}) can be computed in terms of the latent roots of the quaternionic Hermitian matrices $\mathbf{F}_{i}$.

The use of high symmetric planar landmark data for characterising quatenion applications opens an interesting perspective in shape theory, in particular the computation of probabilities are tractable expressions which can be implemented easily. Finally, under these symmetries, (\ref{z}) shows a way of avoiding the quaternions by performing a study based on block $2\times2$ complex matrices, a real normed division algebra more easily understood and handled because its commutative property. These aspects are taking part of a future work.

\section{Conclusions}
This work has set the multimatrix and multimatric variate distributions in a unified approach for the real normed division algebras. The distributions are also indexed by the class of elliptical contoured models. The main advantages of the proposed theory are in agreement with the current paradigms of the distribution theory: 1) The distributions are computable in a simple PC. 2) After integrations, the results can be seen as joint distributions of several combinations of  scalars, vectors and matrix variates, some of them invariant under the family of matrix variate elliptically contoured distributions. An ideal property for situations where the marginals and joint distributions are completely unknown. 3) The multimatrix variate and multimatricvariate distributions share the philosophy of copula theory, but without the restriction to reals, vectors and likelihood copula parameter estimation based on independent distributions. 4) The join distributions can be seen as likelihood functions of probabilistically dependent matrices, as a more real alternative a likelihood function of independent sample variate. 5) The multimatrix variate and multimatricvariate distributions emerged into as unified point of view for all the real normed division algebras, just modulated by a parameter $\beta=1,2,4,8$. 6) The properties presented here are valid for all real normed division algebras, then several applications can be switched according to the sample dependent origin. Finally, a application of symmetric landmark data popularised in real shape theory is translated into the quaternion setting. Current a research about multiple computation of probabilities on symmetric cones is considered.  



\end{document}